\documentclass[a4paper,11pt]{amsart}\usepackage{lmodern}\usepackage{color}\usepackage[T1]{fontenc}\usepackage[applemac]{inputenc}\usepackage[british,francais]{babel}\addto\captionsfrench{}\addto\captionsfrench{}\usepackage{amssymb,amsfonts,amsthm,amsmath,latexsym}\usepackage{enumerate, caption, array, esint}\usepackage{hyperref, etoolbox, soul}\usepackage{ifthen, enumitem, mathrsfs}\setlist[enumerate,1]{label={\normalfont(\roman*)}}\usepackage[margin=1.1in]{geometry}\theoremstyle{plain}\newtheorem{theoreme}{Theorème}[section]\newtheorem{coro}[theoreme]{Corollaire}\newtheorem{lemme}[theoreme]{Lemme}\newtheorem{prop}[theoreme]{Proposition}\theoremstyle{definition}\theoremstyle{remark}\newtheorem{remarque}[theoreme]{Remarque}\newcommand\numberthis{\stepcounter{equation}\tag{\theequation}}\newcommand{\ssum}[1]{\sum_{\substack{#1}}}\newcommand{\e}{{\rm e}}\newcommand{\dd}{{\rm d}}\newcommand{\ee}{{\varepsilon}}\newcommand{\R}{{\mathbb R}}\newcommand{\bbC}{{\mathbb C}}\newcommand{\bbE}{{\mathbb E}}\newcommand{\bbN}{{\mathbb N}}\newcommand{\bbP}{{\mathbb P}}\newcommand{\PP}{{\mathbb P}}\newcommand{\bbR}{{\mathbb R}}\newcommand{\bbZ}{{\mathbb Z}}\newcommand{\bfUn}{{\mathbf 1}}\newcommand{\cA}{{\mathcal A}}\newcommand{\cU}{{\mathcal U}}\newcommand{\Yu}{{\Upsilon_1}}\newcommand{\Yd}{{\Upsilon_2}}\newcommand{\Yt}{{\Upsilon_3}}\newcommand{\Yq}{{\Upsilon_4}}\newcommand{\Yc}{{\Upsilon_5}}\newcommand{\Ys}{{\Upsilon_6}}\newcommand{\Ysp}{{\Upsilon_7}}\newcommand{\gR}{{\mathfrak R}}\newcommand{\GG}{{\mathscr G}}\newcommand{\gT}{\mathfrak T}\newcommand{\gS}{\mathfrak S}\renewcommand{\leq}{\leqslant}\renewcommand{\geq}{\geqslant}\newcommand{\vth}{{\vartheta}}\newcommand{\vphi}{{\varphi}}\renewcommand{\rho}{\varrho}\newcommand{\omegat}{\omega}\newcommand{\infra}{\textit{infra}}\def\sfrac#1#2{{\scriptstyle\frac#1#2}}\newcommand{\ubar}{{\bar u}}\newcommand{\bsigma}{{\bar \sigma}}\newcommand{\floor}[1]{{\left\lfloor {#1} \right\rfloor}}\renewcommand\Re{\operatorname{\mathfrak{Re}}}\newcommand{\Prob}{\bbP}\newcommand{\mes}{\nu_{x, y}}\newcommand{\1}{{\bf 1}}\def\pnu{p^\nu}\def\NO#1{\left\Vert#1\right\Vert}\def\abs#1{\left|#1\right|}\def\sset{\smallsetminus}\newcommand{\dm}{\tfrac12}\newcommand{\dsp}{\displaystyle}\newcounter{constantes}

\newcommand{\cte}[2][]{
\ifthenelse{\expandafter\isundefined\csname constante#2\endcsname}
{
  \stepcounter{constantes}
  \expandafter\xdef\csname constante#2\endcsname{c_{\theconstantes \ifstrempty{#1}{}{, #1}}}
}{}
\csname constante#2\endcsname
}

\def\cD{{\mathscr D}}\newcommand{\cE}{{\mathscr E}}\DeclareMathOperator{\argmin}{argmin}\numberwithin{equation}{section}

\title{Lois de r\'{e}partition des diviseurs des entiers friables}
\author{Sary Drappeau \& Gérald Tenenbaum}

\address{\hspace{-0.3in} Université d'Aix-Marseille, CNRS, Centrale Marseille \\ I2M UMR 7373 \\ 13453 Marseille, France}
\email{sary-aurelien.drappeau@univ-amu.fr}

\address{\hspace{-0.3in} Institut Élie Cartan, Université de Lorraine \\ BP 70239 \\  54506 Vand\oe{}uvre-lès-Nancy Cedex, France}
\email{gerald.tenenbaum@univ-lorraine.fr}


\begin{document}

\begin{otherlanguage}{british}
\begin{abstract}
According to a general probabilistic principle, the natural divisors of friable integers (i.e.~free of large prime factors) should normally present a Gaussian distribution. We show that this indeed is the case with conditional density tending to 1 as soon as the standard necessary conditions are met. Furthermore, we provide explicit, essentially optimal estimates for the decay of the involved  error terms. The size of the exceptional set is sufficiently small to enable recovery of the average behaviour in the same optimal range. Our argument combines the saddle-point method with   new large deviations estimates for the distribution of certain additive functions.
\end{abstract}
\end{otherlanguage}

\maketitle

\thispagestyle{empty}

\section{Introduction}

 La théorie analytique des nombres  explore de manière privilégiée les rapports entre les structures additive et multiplicative des entiers. À ce titre, l'étude de la répartition des diviseurs est emblématique du domaine. Elle a conduit à des avancées significatives et permis de mettre des notions pertinentes en lumière. Outre l'article de synthèse~\cite{Tenenbaum-ErdosUnconv} et la monographie~\cite{HallTenenbaum},  mentionnons les travaux relatifs à la fonction~$\Delta$ de Hooley (voir en particulier~\cite{Hooley, T-DeltaSurvey}) et certaines de ses applications~\cite{BT-Manin}, et les progrès successifs~\cite{Tenenbaum-H, Ford} ayant abouti à la détermination de l'ordre de grandeur dans le problème de la table de multiplication d'Erd\H os~\cite{ErdosUnconv}.
\par 
Soit $n$ un  entier naturel et $\tau(n)$ le nombre de ses diviseurs. Le problème de la répartition des diviseurs de~$n$ est équivalent à l'étude de la variable aléatoire $D_n$ définie par
$$ \PP(D_n = \log d)=\frac1{\tau(n)}\qquad (d\mid n). $$
 La symétrie des diviseurs de~$n$ autour de~$\sqrt{n}$ implique que~$D_n$ est répartie de façon symétrique par rapport à la moyenne~$\bbE(D_n)=\tfrac12\log n$. L'ordre moyen de $D_n$ a été déterminé dans~\cite{DDT}: nous avons 
 $$ \frac1x\sum_{n \leq x} \Prob\big(D_n \leq v\log n\big) = \frac2\pi\arcsin\sqrt{v} + O\bigg(\frac1{\sqrt{\log x}}\bigg)\qquad (x\geqslant 2) $$
uniformément pour~$v\in[0, 1]$. Des généralisations sont proposées et étudiées dans~\cite{BM-DDT}, \cite{BM15}, \cite{GT-add}, \cite{BT16}. \par 
La régularité du comportement en moyenne occulte cependant d'importantes fluctuations du comportement normal: selon le principe d'incertitude établi dans~\cite{Tenenbaum-Div2}, toute suite d'entiers~$\cA\subset \bbN$ pour laquelle~$(D_n/\log n)_{n\in\cA}$ converge en loi est de densité naturelle nulle, autrement dit
$$ |[1, x]\cap \cA| = o(x) \qquad (x\to\infty). $$
\par
Ce phénomène reflète l'influence  des grands facteurs premiers de~$n$ sur la loi de $D_n$. Dans le présent travail, nous examinons le cas où  l'entier~$n$ est astreint à ne posséder que des petits facteurs premiers. 
\par 
Désignons par~$P(n)$ le plus grand facteur premier d'un entier naturel~$n$, avec la convention~$P(1)=1$. On dit que $n$ est~$y$-friable si $P(n)\leqslant y$ et l'on note
$$ S(x, y) := \{n\leq x:  P(n) \leq y\}, \quad\Psi(x, y) := |S(x, y)|. $$
Les quantités
$$ u := (\log x)/\log y, \qquad \ubar := \min\{ u, \pi(y)\}=\min(y,\log x)/\log y\qquad (x\geqslant y\geqslant 2) $$
interviennent classiquement dans l'étude de la répartition des entiers friables. Il découle ainsi des travaux de
Hildebrand~\cite{Hildebrand} et Hildebrand-Tenenbaum \cite{HT} que l'on a, pour tout $\varepsilon>0$ fixé,
$$ \Psi(x, y) = x\varrho(u)\exp\bigg\{O\bigg(\frac{\log 2u}{\log y}+\frac{u}{
\e^{(\log y)^{3/5-\varepsilon}}}
\bigg)\bigg\} \qquad \big((\log x)^{1+\varepsilon}\leq y\leq x\big)$$
où $\varrho$ désigne  la fonction de Dickman, définie comme la solution continue sur~$\bbR^+$ de l'équation fonctionnelle
$$ \varrho(u-1) + u\varrho'(u) = 0 \qquad (u>1) $$
avec la condition initiale~$\varrho(u)=1$ ($0\leqslant u\leq 1$). La fonction $\varrho$ est positive, strictement décroissante sur $[1,+\infty[$ et vérifie $$ \varrho(u) = u^{-u + o(u)} \qquad (u\to\infty). $$
Ainsi~$\varrho(u)\gg 1$ lorsque~$u$ est borné, ce qui implique, d'après le principe d'incertitude mentionné plus haut, que, dans ce cas, la variable  aléatoire~$(D_n/\log n)_{n\in \cA}$ ne converge sur aucun sous-ensemble~$\cA$ de $S(x, y)$ de densité positive.
\par 
Choisir un diviseur aléatoire de~$n$ revient à choisir, indépendamment pour chaque facteur $p^\nu\|n$, un exposant~\hbox{$j\in\{0, \dotsc, \nu\}$} avec équiprobabilité. Nous avons donc la décomposition canonique
\begin{equation}
D_n = \sum_{p^\nu \| n} D_{p^\nu}\label{eq:Dn-xi}
\end{equation}
où les variables~$(D_{p^\nu})$ sont indépendantes, et~$D_{p^\nu}$ est de loi uniforme sur~$\{j \log p:0\leqslant j\leqslant \nu\}$.
\par 
Nous introduisons les moments cumulés centrés
\begin{equation}
m_{k, n} := \sum_{p^\nu\|n}\bbE\Big(\big\{D_{p^\nu} - \dm\nu \log p\big\}^k\Big)\qquad (k\geqslant 1),\label{eq:def-Mk}
\end{equation}
et l'écart-type $\sigma_n$, défini par
$$ \sigma_n^2 =m_{2, n}= \tfrac1{12}\sum_{p^\nu \| n} \nu(\nu+2)(\log p)^2. $$
La symétrie de~$D_{p^\nu}$ autour de sa moyenne implique~$m_{k, n}=0$ lorsque~$k$ est impair.
\par 
 En cohérence avec la décomposition~\eqref{eq:Dn-xi},  nous obtenons une estimation de type limite centrale pour la variable~$D_n$. Il est connu qu'un tel résultat dépend d'une minoration de la variance~$\sigma_n$, associée à une majoration d'un moment~$m_{2k, n}$, $k\geq 2$. Nous posons
\begin{equation}
\Phi(z) :=\frac{1}{\sqrt{2\pi}} \int_z^\infty\e^{-t^2/2} \dd t, \qquad w_n := \frac{\sigma_n^4}{m_{4, n}} \qquad (z\in\R, n\neq 1),\label{eq:defs-Phi-wn}
\end{equation}
et convenons que~$w_1 := 1$. Nous avons ainsi~$w_n\geq \tfrac59$ pour tout entier $n\geqslant 1$ --- cf. formule~\eqref{eq:mkn-mino-expl} \textit{infra}.
\begin{theoreme}\label{thm:TCL-cR}
Il existe une constante absolue~$c>0$, telle que, pour $2\leq y \leq x$, la relation
\begin{equation}
\Prob\Big(D_n \geq \tfrac12\log n + z\sigma_n\Big) = \Phi(z) + O\bigg(\frac{1+z^4}{w_n}\Phi(|z|)\bigg) \qquad \Big(z\ll w_n^{1/4}\Big) \label{eq:TCL-cR}
\end{equation}
ait lieu uniformément tous les entiers $n$ de~$S(x, y)$ sauf au plus~$\ll\e^{-c\sqrt{\ubar}}\Psi(x, y)$ exceptions. 
\end{theoreme}
Nous déduisons de ce résultat et de la minoration de~$w_n$  obtenue à la Proposition~\ref{prop:taille-cR} {\it infra} l'évaluation explicite
suivante.
\begin{coro}\label{thm:pcp-coro}
Soit~$c>0$. La relation
\begin{equation}
\Prob\Big(D_n \geq \tfrac12\log n + z\sigma_n\Big) = \Phi(z) + O\bigg(\frac{1+z^4}{\ubar}\Phi(|z|)\bigg) \qquad\Big(z \ll \ubar^{1/4}\Big)\label{eq:pcp}
\end{equation}
a lieu uniformément pour $x\geqslant y\geqslant 2$ et tous les entiers $n$ de~$S(x, y)$ sauf au plus~$\ll \e^{-c\ubar^{1/4}}\Psi(x, y)$ exceptions.
\end{coro}
Dans ce dernier énoncé, le terme d'erreur et l'uniformité en~$z$ correspondent au domaine naturel de l'approximation gaussienne et sont donc conjecturalement optimaux. Au prix d'un affaiblissement du terme d'erreur et d'une restriction du domaine d'uniformité en $z$, il est toutefois possible d'étendre le champ de validité de l'estimation.
\begin{coro}
\label{thm:pcp}
Il existe une constante absolue $c>0$ telle que, sous les conditions~$2\leq y \leq x$, $1 \leq Z \leq \ubar^{1/4}$, et $z \ll Z$, la relation
$$ \Prob\Big(D_n \geq \tfrac12\log n + z\sigma_n\Big) = \Phi(z) + O\bigg(\frac{1+z^4}{Z^4}\Phi(|z|)\bigg) $$
ait lieu uniformément pour tous les entiers~$n$ de~$S(x, y)$ sauf au plus~$\ll\e^{-c\sqrt{\ubar}/Z}\Psi(x, y)$ exceptions.
\end{coro}
\par 
La validité des deux énoncés précédents résulte immédiatement du Théorème~\ref{thm:TCL-cR} et de la proposition suivante, qui est elle-même conséquence du Corollaire \ref{coro:tailles-Mk} {\it infra}.
\begin{prop}\label{prop:taille-cR}
Il existe une constante absolue $c>0$ telle que, sous les conditions~$2\leq y \leq x$ et~$0 \leq Z \leq c\ubar^{1/4}$, l'inégalité
$$ w_n \geq Z^4 $$
ait lieu pour tous les entiers~$n$ de~$S(x, y)$ sauf au plus $\ll\e^{-c\sqrt{\ubar}/(Z+1)}\Psi(x, y)$ exceptions.
\end{prop}
\par 

Lorsque $z$ est fixé, le terme d'erreur du Théorème~\ref{thm:TCL-cR} est $\ll1/w_n$. Une telle majoration est hors de portée des  théorèmes généraux probabilistes relatifs aux grandes déviations, comme par exemple~\cite[th.~VIII.2, p. 219]{Petrov}. En effet, convenablement adapté et toutes choses égales par ailleurs, cet énoncé ne fournirait, pour le terme d'erreur de~\eqref{eq:pcp}, qu'une borne du type
$$ \ll\bigg(\frac{1 + z}{\sqrt{\ubar}} + \frac{z^4}{\ubar}\bigg) \Phi(|z|). $$
Il est connu que toute amélioration du terme résiduel standard nécessite d'exclure, pour les variables aléatoires indépendantes sommées,   la possibilité d'une concentration au voisinage d'un même réseau unidimensionnel : voir la condition~(III), p.~173, de Petrov~\cite{Petrov} ou la condition~(6.6), p.~547, de Feller~\cite{Feller}. Dans notre cas, la validation de cette hypothèse supplémentaire consiste en une majoration de sommes trigonométriques portant sur les diviseurs premiers de~$n$.  Cette estimation constitue une part significative de notre analyse.

\medskip

Notre résultat fait suite aux travaux de Basquin~\cite{Basquin} et de Drappeau~\cite{D-moyenne}, dévolus à l'étude de la moyenne

\begin{equation}
\label{def-Dxyz}
 \cD(x, y ; z) := \frac1{\Psi(x, y)}\sum_{n\in S(x, y)}\Prob\Big(D_n \geq \tfrac12\log n + z\bsigma\Big) 
 \end{equation}
où~$\bsigma=\bsigma(x, y)$ est  un facteur de renormalisation convenable indépendant de~$n$. Le choix optimal de~$\bsigma$ fait intervenir le point-selle~$\alpha=\alpha(x, y)$ introduit par Hildebrand et Tenenbaum~\cite{HT} dans l'étude de la fonction $\Psi(x,y)$, et dont nous rappelons la définition en~\eqref{eq:implicite-alpha} \infra. Nous avons alors 
\begin{equation}
\label{def-sigmabar}
 \bsigma(x, y)^2 := \frac12\sum_{p\leq y}\frac{p^\alpha - 1/3}{(p^\alpha - 1)^2}(\log p)^2. 
  \end{equation}
Cette définition est en cohérence avec le théorème de Tur\'{a}n-Kubilius friable~\cite[th.~1.1]{BT-TK}, qui implique
$$ \sigma_n =\bigg\{1 + O\bigg(\frac{1}{\ubar^{1/3}}\bigg)\bigg\}\bsigma $$
pour tous les entiers $n$ de $S(x, y)$ sauf au plus~$\ll\Psi(x, y)/\ubar^{1/3}$ exceptions.

L'estimation principale obtenue dans~\cite{D-moyenne} est 
\begin{equation}
\cD(x, y ; z) = \Phi(z) + O_\ee\bigg(\Phi(|z|)\frac{1+z^4}{\ubar}\bigg) \qquad \Big(2\leq y \leq x^{1/(\log_2 x)^{1+\ee}},\,z\ll \ubar^{1/4}\Big), \label{eq:estim-moy}
\end{equation}
où, ici et dans la suite, nous notons~$\log_k x$ la $k$-ième itérée de la fonction logarithme.
Fondé sur la méthode du col en deux variables, l'argument conduisant à \eqref{eq:estim-moy} est spécifique de l'estimation en moyenne et ne fournit pas d'information «presque sûre» de type~\eqref{eq:pcp}. Combiné avec une étude des fluctuations de~$\sigma_n$ autour de son ordre normal~$\bsigma$,  le Théorème~\ref{thm:pcp-coro}  fournit naturellement l'estimation suivante.
\begin{theoreme}\label{thm:Dxyz}
Nous avons
\begin{equation}
\cD(x, y ; z) = \Phi(z) + O\bigg(\Phi(|z|)\frac{1+z^4}{\ubar}\bigg) \qquad \Big(2\leq y \leq x,\ |z|\leq \cte{moy}\ubar^{1/5}\Big),\label{eq:estim-moy-2}
\end{equation}
où ~$\cte{moy}>0$ est une constante convenable. 
\end{theoreme}
Le domaine d'uniformité en~$z$ est moins grand que celui de~\eqref{eq:estim-moy} mais le résultat vaut sans contrainte sur~$x$ et~$y$.


\subsection*{Notations}
Nous désignons par~$\mes$ la probabilité conditionnelle uniforme sur~$S(x, y)$ : pour toute partie~$\cA \subset [1,x]$, nous avons donc
$$ \mes(\cA) := \frac{|\cA\cap S(x,y)|}{\Psi(x, y)}. $$
Par ailleurs, nous notons classiquement~$\omega(n)$ le nombre des facteurs premiers distincts d'un entier naturel~$n$.
\par 
Enfin, sauf dépendance explicitement mentionnée, les constantes $c,\,c_1,\,c_2,$ etc. apparaissant dans ce travail sont absolues et strictement positives.


\section{Entiers friables}

\subsection{Méthode du col}

Notre résultat est basé sur l'utilisation de la méthode du col, développée dans ce contexte par Hildebrand et Tenenbaum~\cite{HT}. Nous rappelons ici les principales estimations  dont nous ferons usage.

La quantité~$\Psi(x, y)$ est représentée par l'intégrale de Perron
$$ \Psi(x, y) = \frac1{2\pi i}\int_{\sigma - i\infty}^{\sigma+i\infty} \zeta(s, y) \frac{x^s\dd s}{s} \qquad (x\not\in\bbN), $$
où intervient le produit eulérien tronqué
$$ \zeta(s, y) := \sum_{P(n) \leq y} \frac1{n^s} = \prod_{p\leq y}\Big(1-\frac1{p^s}\Big)^{-1}. $$
Le point-selle~$\alpha = \alpha(x, y)$ est défini par~$\alpha := \argmin_{\sigma > 0} x^\sigma \zeta(\sigma, y)$. Il s'agit donc de la fonction implicite correspondant à l'équation 
\begin{equation}
\sum_{p\leq y}\frac{\log p}{p^\alpha - 1} = \log x.\label{eq:implicite-alpha}
\end{equation}
Observons immédiatement, à fins de référence ultérieure, que l'équation \eqref{eq:implicite-alpha} implique immédiatement
\begin{equation}
\label{minya}
y^\alpha\gg 1+\frac{y}{\log x}\cdot
\end{equation}
Définissons également
\begin{equation}
\sigma_2^* = \sigma_2^*(\alpha, y) := \sum_{p\leq y}\frac{(\log p)^2 p^\alpha}{(p^\alpha - 1)^2}\cdot\label{eq:def-sigma2}
\end{equation}
 L'évaluation par le théorème des nombres premiers des sommes en~$p$ de~\eqref{eq:implicite-alpha} et~\eqref{eq:def-sigma2} (\textit{cf.}~\cite[th.~2]{HT}) fournit les estimations
\begin{eqnarray}\label{eq:approx-alpha}
 &\alpha = \dsp\frac{\log(1+y/\log x)}{\log y}\Big\{ 1 + O\Big(\frac{\log_2 2y}{\log y}\Big)\Big\},
\\
\label{eq:approx-sigma2}
&\sigma_2^* = \dsp\Big(1+\frac{\log x}y\Big)(\log x)(\log y)\Big\{1 + \Big(\frac1{\log 2\ubar}\Big)\Big\}.
\end{eqnarray}
Mentionnons aussi l'approximation 
\begin{equation}
\label{eq:estim-alpha-1111}
\alpha(x, y) = 1 - \frac{\log(\ubar \log(\ubar+1)) + O(1)}{\log y} \qquad \big(2\leq y \leq x\big).
\end{equation}
\par 
Le résultat principal de~\cite{HT} est la formule asymptotique uniforme
$$ \Psi(x, y) = \frac{x^\alpha \zeta(\alpha, y)}{\alpha\sqrt{2\pi \sigma_2^*}}\Big\{1 + O\Big(\frac1{\ubar}\Big)\Big\}\qquad (2\leqslant y\leqslant x), $$
qui implique en particulier
\begin{equation}\label{eq:psi-zeta-asymp}
\Psi(x, y) \asymp \frac{x^\alpha \zeta(\alpha, y)\sqrt{\ubar}}{\alpha\log x} \qquad \big(2\leq y \leq x\big).
\end{equation}
La majoration
\begin{equation}\label{eq:psi-local}
\Psi\Big(\frac xd, y\Big) \ll \frac{\Psi(x, y)}{d^\alpha } \qquad \big(2\leq y\leq x,\, 1\leq d \leq x\big),
\end{equation}
qui est un cas particulier important de~\cite[th.~2.4]{BT05}, nous sera également utile.
\subsection{Moyennes friables de fonctions multiplicatives}

Divers travaux, notamment~\cite[cor.~2.3]{TW}, permettent des estimations satisfaisantes de la valeur moyenne d'une fonction multiplicative~$g\geq 0$ sous des hypothèses relativement générales concernant le comportement en moyenne des valeurs~$g(p)$, dans un domaine de la forme
$$ \e^{(\log_2 x)^{5/3+\ee}} \leq y \leq x .\leqno{(H_\varepsilon)}$$
Pour des valeurs de~$y$ plus petites, la méthode du col demeure théoriquement applicable. Cependant, à mesure que~$\alpha(x, y)$ diminue, la description du comportement moyen de~$g$ nécessite des hypothèses sur les valeurs~$g(p^\nu)$ pour des valeurs croissantes de~$\nu$ --- voir par exemple \cite[th.~1.4]{BT-TK} et les commentaires qui suivent cet énoncé. La recherche d'estimations  uniformes en~$g$ devient alors un problème significativement plus délicat.
\par 
Dans le cadre du présent travail, une majoration de la valeur moyenne
$$ G(x, y) := \frac1{\Psi(x, y)}\sum_{n\in S(x, y)} g(n) $$
est suffisante.
Dans cette perspective, nous pouvons  nous contenter d'une approche élémentaire combinée avec l'évaluation au point-selle~\eqref{eq:psi-zeta-asymp} et la majoration semi-asymptotique~\eqref{eq:psi-local}. Notre estimation s'exprime en termes du produit eulérien
$$ \GG_y(\alpha) := \frac1{\zeta(\alpha, y)} \prod_{p\leq y}\ssum{\nu\geq 0 \\ p^\nu \leq x}\frac{g(p^\nu)}{p^{\nu\alpha}}, $$
où~$x$, $y$ et~$\alpha$ sont liés par la relation~\eqref{eq:implicite-alpha}. Lorsque, pour chaque nombre premier $p\leqslant y$, il existe $\nu_p\in\bbN^*$ tel que
\begin{equation}
g(p^\nu)=0\quad (\nu>\nu_p), \qquad g(p^{\nu-1}) \leq g(p^\nu) \qquad (1\leq \nu\leq \nu_p) ,\label{eq:cond-croiss}
\end{equation}
nous posons
$$ \GG_y^*(\alpha) := \frac1{\zeta(\alpha, y)} \prod_{p\leq y}\bigg\{\sum_{ 0 \leqslant \nu\leqslant \nu_p}\frac{g(p^\nu)}{p^{\nu\alpha}}+\frac{g(p^{\nu_p})}{p^{\nu_p\alpha}(p^\alpha-1)}\bigg\}=\prod_{p\leq y}\bigg\{1+\sum_{1\leqslant \nu\leqslant \nu_p}\frac{g(\pnu)-g(p^{\nu-1})}{p^{\nu\alpha}}\bigg\}.$$

\begin{lemme}
\label{lemme:majo-f-mult}
Soit~$g:\bbN\to\bbR^+$ une fonction multiplicative positive ou nulle. Les assertions suivantes sont valables pour~$2\leq y \leq x$.
\begin{enumerate}
\item Si~$g(n) \leq \kappa^{\omega(n)}$ pour un certain~$\kappa\geq 1$ et tout entier~$n\geqslant 1$, alors
\begin{equation}
G(x, y) \ll \kappa\sqrt{\ubar}(\log2\ubar)^2 \GG_y(\alpha) .\label{eq:majo-G-1}
\end{equation}
\item Sous la condition \eqref{eq:cond-croiss}, nous avons
\end{enumerate}
\begin{equation}
G(x, y) \ll \GG_y^*(\alpha) .\label{eq:majo-C-2}
\end{equation}
\end{lemme}
\begin{remarque}
Il est essentiel pour notre application de disposer d'une majoration de type~\eqref{eq:majo-G-1} dans laquelle le facteur de~$\GG_y(\alpha)$ est~$\e^{o(\ubar)}$ lorsque~$\ubar\to\infty$.
\end{remarque}
\begin{proof}
Supposons dans un premier temps~$\log y \leq (\log 2u)^2$, et donc~$\log y \ll (\log 2\ubar)^2$. La positivité de~$g$ permet d'écrire, compte tenu de~\eqref{eq:psi-zeta-asymp},
$$ G(x, y) \leq \frac{x^\alpha}{\Psi(x, y)} \sum_{n\in S(x, y)}\frac{g(n)}{n^\alpha} \ll \frac{\sqrt{\ubar}(\log y)}{\zeta(\alpha, y)}\prod_{p\leq y} \ssum{\nu\geq 0\\ p^\nu \leq x} \frac{g(p^\nu)}{p^{\nu\alpha}}\cdot $$
Cela implique bien~\eqref{eq:majo-G-1} dans l'éventualité considérée.
\par 
Supposons maintenant~$\log y > (\log 2u)^2$, et donc~$y\gg \log x$. Nous utilisons une inégalité de type Rankin pour la quantité\begin{equation}
G(x, y)\log x = \frac1{\Psi(x, y)}\sum_{n\in S(x, y)} g(n)\Big\{\log n + \log\Big(\frac xn\Big)\Big\}.\label{eq:G-log-decomp}
\end{equation}
Puisque~$\log(x/n)\leq (x/n)^\alpha/\alpha$, nous avons
\begin{equation}
\sum_{n\in S(x, y)} g(n)\log\Big(\frac xn\Big) \leqslant \frac{x^\alpha}{\alpha} \sum_{n\in S(x, y)} \frac{g(n)}{n^\alpha} \ll \frac{\log x}{\sqrt{u}} \Psi(x, y)\GG_y(\alpha) .\label{eq:G-apdecomp-1}
\end{equation}
Considérons ensuite la contribution du terme~$\log n$ au membre de droite de~\eqref{eq:G-log-decomp}. Sous l'hypothèse $g(n)\leqslant \kappa^{\omega(n)}$, nous avons
\begin{align*}
\sum_{n\in S(x, y)} g(n)\log n = &\ \sum_{p^\nu \in S(x, y)}\ssum{m\in S(x/p^\nu, y) \\ p\,\nmid\, m} g(mp^\nu)\log(p^\nu) \\
\leq &\ \kappa \sum_{m\in S(x, y)} g(m) \sum_{p^\nu \in S(x, y)} \log(p^\nu).
\end{align*}
La somme intérieure est 
$$ \ll \sum_{p\leq \min\{y, x/m\}} \frac{\{\log(x/m)\}^2}{\log p} \ll \min\bigg\{ \frac xm, \frac{y(\log(x/m))^2}{(\log y)^2}\bigg\} \leq \Big(\frac xm\Big)^{\alpha}(yu^2)^{1-\alpha}. $$
Comme la condition~$\log2u < \sqrt{\log y}$ implique
$$ (1-\alpha)\log u \ll (\log 2u)^2/\log y \ll 1, $$
nous obtenons
\begin{equation}
\begin{aligned}
\sum_{n\in S(x, y)} g(n)\log n \ll &\ \kappa y^{1-\alpha} x^\alpha \sum_{m\in S(x, y)}\frac{g(m)}{m^\alpha} \\ \ll &\ \kappa\sqrt{u}(\log 2u)(\log x)\Psi(x, y) \GG_y(\alpha).
\end{aligned}
\label{eq:G-apdecomp-2}
\end{equation}
Reportons~\eqref{eq:G-apdecomp-1} et~\eqref{eq:G-apdecomp-2} dans~\eqref{eq:G-log-decomp}.  Après développement de la somme sur $S(x,y)$ en produit eulérien, il suit
$$ G(x, y) \ll \Big(\kappa(\log 2u)  \sqrt{u}+ \frac1{\sqrt{u}}\Big)\GG_y(\alpha). $$
Cela implique~\eqref{eq:majo-G-1} dans le domaine~$\log y \geq (\log 2u)^2$ et conclut donc la preuve de l'assertion~(i).
\par 
Pour établir~\eqref{eq:majo-C-2}, nous introduisons la fonction arithmétique multiplicative $h$ telle que ~$g=\1*h$. Pour tout entier~$n$ de $E_g := \{n\in\bbN^* : g(n)>0\}$, l'hypothèse~\eqref{eq:cond-croiss} implique alors~$h(n)\geq 0$ et~$d\in E_g$ pour tout~$d|n$.
Ainsi
$$ G(x, y) = \frac1{\Psi(x, y)}\ssum{n\in E_g\cap S(x, y)} g(n) \leq \sum_{d\in E_g\cap S(x, y)} h(d)\frac{\Psi(x/d, y)}{\Psi(x, y)} \ll \sum_{d\in E_g\cap S(x, y)}\frac{h(d)}{d^\alpha} $$
où l'on a fait appel à~\eqref{eq:psi-local}. Un calcul de produit eulérien permet ensuite, grâce à la positivité de~$g$, de majorer la dernière somme en $d$ par $\GG_y^*(\alpha)$.
Cela termine la preuve de l'assertion~(ii).
\end{proof}


\section{Répartition des valeurs de certaines fonctions additives}

Nous nous proposons ici d'étudier l'ordre de grandeur normal friable des fonctions additives
$$ f_k(n) := \ssum{p^\nu \| n \\ \nu > 0} (\nu\log p)^k $$
lorsque l'entier~$k\geq 0$ est fixé. Un outil essentiel à cette tâche est l'inégalité de Tur\'{a}n-Kubilius (voir~\cite[ch.~III.3]{ITAN}), dont l'analogue friable a été développé dans des travaux récents de La Bretèche et Tenenbaum~\cite{BT-TK, BT-TK2}. Ces travaux exploitent une modélisation de la structure multiplicative de l'ensemble~$S(x, y)$ consistant en l'approximation $$\nu_{x,y}\{n\geqslant 1:\pnu\|n\}\approx (1-p^{-\alpha})/p^{\nu\alpha},$$ les événements étant considérés comme indépendants pour des valeurs distinctes de $p\leq y$.
\par 
Cette heuristique est confortée par le fait, établi dans~\cite{BT-TK}, que les valeurs prises par~$f_k$ ont tendance à se concentrer autour de la valeur moyenne du modèle
\begin{equation}
A_{f_k}(x, y) := \ssum{p\leq y \\ \nu\geq 1} \frac{f_k(p^\nu)}{p^{\nu\alpha}}\Big(1-\frac1{p^\alpha}\Big),\label{eq:def-Af}
\end{equation}
lorsque~$n\in S(x, y)$ et~$\ubar \to \infty$. Notons qu'il existe une suite de polynômes $\{Q_k\}_{k=0}^{\infty}$ à coefficients positifs ou nuls telle que $\deg Q_k\leqslant k$ et
 $$A_{f_k}(x,y)=\sum_{p\leqslant y}\Big(1-\frac{1}{p^\alpha}\Big)\frac{(\log p)^kQ_k(p^\alpha)}{(p^\alpha-1)^{k+1}},$$ d'où l'on déduit comme au lemme~4 de~\cite{HT} l'évaluation explicite
\begin{equation}
\label{eq:asymp-Afk}
A_{f_k}(x, y) \asymp_k (\log x)^k/\ubar^{k-1} \qquad (k\geq 1).
\end{equation}\par 
Nous nous intéressons dans ce qui suit à estimer la taille des sous-ensembles de~$S(x, y)$ caractérisés par la propriété que~$f_k(n)$ est significativement plus grand ou plus petit que la moyenne~$A_{f_k}(x, y)$.

\subsection{Nombre des facteurs premiers}

Le cas~$k=0$, qui correspond à~$f_0(n) = \omega(n)$, est particulier. Nous avons, d'après les lemmes~3.2, 3.6 et l'équation~(2.37) de~\cite{BT05},
\begin{align*}
A_{\omega}(x, y) = \sum_{p\leq y} \frac1{p^\alpha} = &\ \Big\{1+O\Big(\frac1{\log y}+\frac1{\log2\ubar}\Big)\Big\}\frac{yu}{y + \log x} + \log_2 y \numberthis\label{eq:asymp-norm-omega} \\
\asymp &\ \ubar + \log_2 y.
\end{align*}
Le terme~$\log_2 y$, qui est dû à l'influence des petits facteurs premiers, n'intervient pas dans l'étude des fonctions~$f_k$ lorsque~$k\geqslant 1$.
\par 
Posons
\begin{equation}
\omega_{x, y}(n) := \big|\{p^\nu \| n: \ \sqrt{y} <p \leq y,\ u/(2\ubar) \leq \nu \leq 2u/\ubar\}\big|\qquad (n\geqslant 1)\label{eq:def-omegat}
\end{equation}
et notons  que $ u/\ubar \asymp 1 + (\log x)/y$.
La proposition suivante permet d'établir que, avec une probabilité tendant vers 1, les  puissances de nombres premiers comptés dans $\omegat_{x, y}(n)$  contribuent significativement au terme~$yu/(y+\log x)$ apparaissant dans le membre de droite de~\eqref{eq:asymp-norm-omega}. 

\begin{lemme}\label{lemme:mino-omegat}
Il existe constante absolue~$\cte{wt}>0$ telle que l'on ait
\begin{equation}
\mes\big\{n\geqslant 1 : \omegat_{x, y}(n) \leq \cte{wt} \ubar\big\} \ll \e^{-\cte{wt}\ubar}\qquad (2\leqslant y\leqslant x).\label{eq:mino-omegat}
\end{equation}
\end{lemme}
\begin{proof}
Pour tout~$c>0$, nous avons
\begin{equation}
\nu_{x, y}\big\{n\geqslant 1 : \omega_{x, y}(n) \leq c \ubar \big\} \leq \frac{\e^{c\ubar}}{\Psi(x, y)}\sum_{n\in S(x, y)}\e^{-\omega_{x, y}(n)}.\label{eq:omegat-nu-exp}
\end{equation}
D'après~\eqref{eq:majo-G-1} avec~$\kappa=1$, le membre de droite ne dépasse pas
\begin{equation}
\begin{aligned}
&\frac{\ubar\e^{c\ubar}}{\zeta(\alpha, y)}\sum_{P(n)\leq y}\frac{\e^{-\omega_{x, y}(n)}}{n^\alpha} = \ubar\e^{c\ubar}\prod_{\sqrt{y}< p \leq y} \bigg\{1 - \Big(1-\frac{1}\e\Big)\bigg(\frac1{p^{\nu_1\alpha}} - \frac1{p^{(\nu_2+1)\alpha}}\bigg)\bigg\}
\end{aligned}\label{eq:omegat-somme-prod}
\end{equation}
où l'on a posé~$(\nu_1, \nu_2) := (\lceil u/(2\ubar)\rceil, \lfloor 2u/\ubar\rfloor )$. Comme~$\nu_1-1\ll1/(\alpha \log y)$ d'après \eqref{eq:approx-alpha}, nous pouvons écrire
$$ p^{\nu_1 \alpha} \ll p^{\alpha}\qquad (\sqrt{y}<p\leqslant y). $$
De plus, nous avons $(\nu_2 - \nu_1 + 1)\alpha \log p \geq \dm\{3u/(2\ubar) - 1\}\alpha\log y \gg 1$ pour toutes les valeurs de $p$ considérées, donc
$$ \frac1{p^{\nu_1\alpha}} - \frac1{p^{(\nu_2+1)\alpha}} \gg \frac1{p^\alpha}\qquad (\sqrt{y}<p\leqslant y). $$
Or,~$\sum_{\sqrt{y} <p \leq y} 1/p^{\alpha} \gg \ubar$ d'après \cite[lemme~3.6]{BT05}. Il existe donc une constante absolue $\cte{cprime}$ telle que le produit en $p$ de \eqref{eq:omegat-somme-prod} ne dépasse pas $\e^{-\cte{cprime}\ubar}$.
L'estimation souhaitée en résulte par report dans~\eqref{eq:omegat-nu-exp} et~\eqref{eq:omegat-somme-prod} pour le choix choix~$\cte{wt}=\cte{cprime}/3$.
\end{proof}

\subsection{Minoration de~\texorpdfstring{$f_k(n)$}{fk}}

Nous déterminons ici une minoration de~$f_k(n)$ valable pour chaque indice~$k\geq 1$ fixé et ``presque'' tous les entiers $n$ de $S(x,y)$. Rappelons la définition \eqref{eq:def-omegat}.
\begin{lemme}\label{lemme:mino-fk}
Soit~$k\geq 1$. Nous avons 
$$ f_k(n) \gg_k (\log x)^k/\ubar^{k-1} $$
pour tous les entiers $n$ de~$S(x, y)$ sauf au plus~$\ll\e^{-\cte{wt}\ubar}\Psi(x, y)$ exceptions.
\end{lemme}
\begin{proof} Nous avons trivialement $f_k(n) \geq  (u\log y/4\ubar)^k\omega_{x,y}(n)$. L'assertion de l'énoncé résulte donc immédiatement du Lemme~\ref{lemme:mino-omegat}.
\end{proof}
Cette majoration simple suffit à la preuve du Théorème~\ref{thm:TCL-cR}. L'estimation plus précise suivante, qui possède un intérêt propre, ne sera pas utilisée dans la suite. En conformité avec \eqref{eq:def-Af}, nous posons en toute généralité
$$A_f(x,y):=\ssum{p\leqslant y \\ \nu\geqslant 1}\frac{f(\pnu)}{p^{\nu\alpha}}\Big(1-\frac{1}{p^\alpha}\Big),\qquad B_f(x,y)^2:=\ssum{p\leqslant y \\ \nu\geqslant 1}\frac{f(\pnu)^2}{p^{\nu\alpha}}\Big(1-\frac{1}{p^\alpha}\Big)$$
lorsque $f$ est une fonction additive.
\begin{prop}\label{prop:mino-fk-precis}
Soient~$k\geq 1$ et~$f$ une fonction additive satisfaisant à
\begin{equation}
0\leq  f(p^\nu) \asymp (\nu\log p)^k\qquad (p\geqslant 2,\nu\geqslant 1).\label{eq:hypo-mino-prec}
\end{equation}
Pour~$2\leq y \leq x$ et~$\delta\in]0, 1]$, nous avons alors
\begin{equation}
\nu_{x, y}\big\{n : f(n) \leq (1-\delta) A_{f}(x, y)\big\} \ll_k \ubar \e^{- \cte[k]{minofk} \delta^2 \ubar}\label{eq:mino-f-precis}
\end{equation}
où~$\cte[k]{minofk}>0$ ne dépend que de~$k$ et de la constante implicite dans~\eqref{eq:hypo-mino-prec}.
\end{prop}
\begin{proof}
Notons $\wp^-(x,y)$ le membre de gauche de \eqref{eq:mino-f-precis}. Pour tout $\lambda\geqslant 0$, nous avons
 \begin{equation}
 \wp^-(x,y)\leqslant \e^{\lambda(1-\delta)A_{f}(x, y)} \sum_{n\in S(x, y)} \e^{-\lambda f(n)}\ll \frac{\ubar \e^{\lambda(1-\delta)A_{f}(x, y)}}{\zeta(\alpha,y)}\prod_{p\leqslant y}\sum_{\nu\geqslant 0}\frac{\e^{-\lambda f(\pnu)}}{p^{\nu\alpha}},
 \end{equation}
 où la majoration résulte de \eqref{eq:majo-G-1} avec $\kappa=1$.
 Majorons le dernier produit à l'aide de l'inégalité $\e^{-v}\leqslant 1-v+\dm v^2$ $(v\geqslant 0)$. Il suit
 \begin{align*}\wp^-(x,y)&\ll \ubar \e^{\lambda(1-\delta)A_{f}(x, y)}\prod_{p\leqslant y}\Big(1-\frac{1}{p^\alpha}\Big)\bigg(\sum_{\nu\geqslant 0}\frac1{p^{\nu\alpha}}-\lambda\frac{f(\pnu)}{p^{\nu\alpha}}+\dm\lambda^2\frac{f(\pnu)^2}{p^{\nu\alpha}}\bigg)\\
 &= \ubar \e^{\lambda(1-\delta)A_{f}(x, y)}\prod_{p\leqslant y}\bigg(1-\lambda \Big(1-\frac{1}{p^\alpha}\Big)\sum_{\nu\geqslant 1}\frac{f(\pnu)}{p^{\nu\alpha}}+\dm\lambda^2\Big(1-\frac{1}{p^\alpha}\Big)\sum_{\nu\geqslant 1}\frac{f(\pnu)^2}{p^{\nu\alpha}}\bigg)\\
 &\leqslant \ubar \e^{-\lambda\delta A_{f}(x, y)+\lambda^2B_f(x,y)^2/2}.
 \end{align*}
Pour le choix optimal $\lambda:=\delta A_f(x,y)/B_f(x,y)^2$, nous obtenons  donc
$$\wp^-(x,y)\ll \ubar\e^{-\delta^2A_f(x,y)^2/2B_f(x,y)^2}.$$
L'estimation annoncée résulte donc de \eqref{eq:asymp-Afk}.
 \end{proof}

\subsection{Majoration de~\texorpdfstring{$f_k(n)$}{fk}}

L'obtention d'une majoration de~$f_k(n)$ valable sur un sous-ensemble dense de $S(x,y)$ est plus délicate que celle de la minoration, traitée au paragraphe précédent. La difficulté technique sous-jacente provient de la vitesse de croissance de~$f_k(p^\nu)$ en fonction de l'exposant~$\nu$. Cela rend nécessaire une troncature préalable des grands exposants de la factorisation.
\begin{prop}\label{prop:majo-fk}
Soit $k\in\bbN^*$ et $f$  une fonction additive satisfaisant à \eqref{eq:hypo-mino-prec} et 
$$f(p^{\nu-1})\leq f(p^\nu) \qquad (p\leq y, \,\nu\geq 1). $$
Il existe une constante~$\cte[k]{majofk}>0$ telle que, sous les conditions~\mbox{$2\leq y \leq x$}, $\delta> 0$, nous ayons uniformément
\begin{equation}
\mes\big\{n : f(n)\geq (1+\delta) A_{f}(x, y)\big\} \ll_{k} \ubar \exp\left\{- 2\cte[k]{majofk} \left(\min\{\delta, \delta^2\} \ubar\right)^{1/k} \right\}.\label{eq:majo-fk-lemme}
\end{equation}
En particulier, pour chaque entier~$k\geq 1$, et uniformément pour~$h\geq 0$, $2\leq y \leq x$, nous avons
$$ f(n) \ll_k \ubar^{1+h-k} (\log x)^k $$
pour tous les entiers $n$ de $S(x,y)$ sauf  au plus~$\ll_k\e^{- \cte[k]{majofk}\ubar^{(1+h)/k}}\Psi(x, y)$ exceptions.
\end{prop}

\begin{proof} Notons $\wp^+(x,y)$ le membre de gauche de \eqref{eq:majo-fk-lemme}. Donnons-nous un paramètre positif $v$ à optimiser ultérieurement et posons $L:=\kappa(\ubar^v/\alpha+\log y)$, où $\kappa$ est une constante assez grande ne dépendant que de $k$. \par 
Notant alors $\nu_p:=\floor{L/\log p}$, nous commençons par estimer la contribution $\wp_1^+(x,y)$ des entiers $n$ qui ne divisent pas $\prod_{p\leqslant y}p^{\nu_p}$. Nous avons trivialement, grâce à \eqref{eq:psi-local},
$$
\wp_1^+(x,y)\leqslant \sum_{p\leqslant y}\frac{\Psi(x/p^{\nu_p+1},y)}{\Psi(x,y)}\ll\sum_{p\leqslant y}\frac1{p^{\alpha(\nu_p+1)}}\ll\frac{\pi(y)}{\e^{\alpha L}}.$$
On vérfie aisément grâce à \eqref{eq:approx-alpha} que $\alpha L>\ubar^v+\log (1+y/\log x)$ dès que $\kappa$ est assez grand. Il suit
\begin{equation}
\label{majP1}
\wp_1^+(x,y)\ll\ubar\,\e^{-\ubar^v}.
\end{equation}
\par 
Considérons à présent la contribution complémentaire, notée $\wp_2^+(x,y)$. Soit $\lambda$ un paramètre positif ou nul. La croissance en $\nu$ de $f(\pnu)$ permet d'appliquer le Lemme \ref{lemme:majo-f-mult}(ii) à la fonction multiplicative $g$ définie par $g(\pnu):=\e^{\lambda f(\pnu)}$ si $\nu\leqslant \nu_p$ et $g(\pnu):=0$ dans le cas contraire. Nous obtenons
\begin{equation}\label{majP2+}
\begin{aligned}
\wp_2^+(x,y)&\leqslant \dsp\frac{\e^{-(1+\delta)\lambda A_{f}(x,y)}}{\Psi(x,y)}\sum_{n\in S(x,y)}g(n)\\
&\ll \e^{-(1+\delta)\lambda A_{f}(x,y)}\prod_{p\leq y}\bigg\{1+\sum_{1\leqslant \nu\leqslant \nu_p}\frac{\e^{\lambda f(\pnu)}-\e^{\lambda f(p^{\nu-1})} }{p^{\nu\alpha}}\bigg\}
\end{aligned}
\end{equation}
La dernière somme en $p$ n'excède pas
$$B_p(\alpha,y):=\lambda\sum_{1\leqslant \nu\leqslant \nu_p }\frac{\{f(\pnu)-f(p^{\nu-1})\}\e^{\lambda f(\pnu)}}{p^{\nu\alpha}}\cdot$$
Soient alors $\eta\in[0,\dm\alpha]$, $\beta:=\alpha-\eta$. Pour le choix  $\lambda:=c\eta/L^{k-1}$ où $c$ est une constante absolue assez petite, nous avons $$\e^{\lambda f(\pnu)}\leqslant p^{\nu\eta}=p^{\nu(\alpha-\beta)}\qquad (1\leqslant \nu\leqslant \nu_p).$$
Il suit
$$B_p(\alpha)\leqslant \lambda \sum_{1\leqslant \nu\leqslant \nu_p }\frac{f(\pnu)-f(p^{\nu-1})}{p^{\nu\beta}}=\lambda\sum_{1\leqslant \nu\leqslant \nu_p }\Big(1-\frac1{p^\beta}\Big)\frac{f(\pnu)}{p^{\nu\beta}}+\frac{\lambda f(p^{\nu_p})}{p^{(\nu_p+1)\beta}}\cdot$$
Le dernier terme  est $\ll\eta L/\e^{\beta L}$. Une application standard du théorème des accroissements finis permet donc d'écrire
 $$B_p(\alpha)\leqslant \lambda \sum_{1\leqslant \nu\leqslant \nu_p }\Big(1-\frac1{p^\alpha}\Big)\frac{f(\pnu)}{p^{\nu\alpha}}+O\bigg(\eta\lambda\sum_{1\leqslant \nu\leqslant \nu_p }\Big(1-\frac1{p^\beta}\Big)\frac{f_{k+1}(\pnu)}{p^{\nu\beta}}+\frac{\eta L}{\e^{\beta L}}\bigg),$$
 d'où
 $$\sum_{p\leqslant y}B_p(\alpha)\leqslant\lambda A_{f}(x,y) +O\bigg(\eta\lambda\gR+\frac{\eta L\pi(y)}{\e^{\beta L}}\bigg)$$
 avec $$\gR:=\sum_{p\leqslant y}\sum_{1\leqslant \nu\leqslant \nu_p }\Big(1-\frac1{p^\beta}\Big)\frac{f_{k+1}(\pnu)}{p^{\nu\beta}}\asymp \sum_{p\leqslant y}\frac{(\log p)^{k+1}Q_k(p^\beta)}{(p^\beta-1)^{k+1}}$$
 où $Q_k$ est un polynôme de degré $k$. Si $\eta=\alpha-\beta\ll \ubar/\log x$, nous avons $p^\beta\asymp p^\alpha$ pour tout $p\leqslant y$. Il s'ensuit que
 $$\gR\ll A_{f_{k+1}}(x,y)\asymp(\log x)^{k+1}/\ubar^k.$$
En reportant dans \eqref{majP2+}, nous obtenons 
\begin{equation}
\wp_2^+(x,y)\ll \e^{-\delta \lambda A_{f}(x,y)+B}
\end{equation}
avec 
$$B\ll \frac{\eta\lambda(\log x)^{k+1}}{\ubar^k}+\frac{\eta L\pi(y)}{\e^{\beta L}}\ll \frac{\eta\log x}\ubar \bigg\{\lambda A_{f}(x,y)+\frac{\ubar L\pi(y)}{\e^{\beta L}\log x}\bigg\}\ll \frac{\eta\log x}\ubar\big\{ \lambda A_{f}(x,y)+1\big\}.$$
Pour le choix $\eta:=c\min(1,\delta)\ubar/\log x$, il vient ainsi 
\begin{equation}
\label{majP2}\wp_2^+(x,y)\ll \e^{-\sfrac12\delta \lambda A_{f}(x,y)}\ll\e^{-\cte{p2}\min(\delta,\delta^2)\ubar^{1-(k-1)v}},
\end{equation}
puisque 
\begin{align*} \lambda A_{f}(x,y)&\gg\min(1,\delta)\ubar\bigg(\frac{\log x}{\ubar L}\bigg)^{k-1}\gg\min(1,\delta)\ubar\bigg(\frac{u\alpha\log y}{\ubar(\ubar^v+\alpha\log y)}\bigg)^{k-1}\gg\min(1,\delta)\ubar^{1-(k-1)v}.
\end{align*}
Il reste à choisir $v$ optimalement en fonction de \eqref{majP2} et\eqref{majP1}, soit $\ubar^v:=\big\{\min(\delta,\delta^2)\ubar\big\}^{1/k}$. Cela implique bien le résultat annoncé.
\end{proof}

\subsection{Tailles de~\texorpdfstring{$m_{2, n}$}{m2} et~\texorpdfstring{$m_{4, n}$}{m4}}
\label{sec:taille-mn}
Les estimations du paragraphe précédent permettent un encadrement des moments~$m_{2, n}$ et~$m_{4, n}$ définis en~\eqref{eq:def-Mk}. Nous avons explicitement
\begin{equation}
m_{2, n} =\tfrac1{12} \sum_{p^\nu\|n} \nu(\nu+2)(\log p)^2,\quad m_{4, n} = \tfrac1{240}\sum_{p^\nu\|n} \nu(\nu+2)(3\nu^2+6\nu-4)(\log p)^4, \label{eq:m2n-expl}
\end{equation}
et notons d'emblée que
\begin{equation}
w_n = \frac{m_{2, n}^2}{m_{4, n}} \geq \frac{\sum_{p^\nu \| n} \nu^2(\nu+2)^2(\log p)^4/144}{\sum_{p^\nu \| n} \nu(\nu+2)(3\nu^2+6\nu-4)(\log p)^4/240} \geq \frac 53 \inf_{\nu\geq 1}\frac{\nu(\nu+2)}{3\nu^2+6\nu-4} = \frac 59\cdot\label{eq:mkn-mino-expl}
\end{equation}
\begin{coro}\label{coro:tailles-Mk}
Il existe une constante absolue~$\cte{mk}>0$ telle que les estimations 
\begin{equation}
\begin{cases} m_{2, n} \asymp (\log x)^2/\ubar, \\ (\log x)^4/\ubar^{3} \ll m_{4, n} \ll (\log x)^4/\ubar^{2} \end{cases}\label{eq:taille-Mk-ps12}
\end{equation}
aient lieu uniformément pour $2 \leq y \leq x$ et tous les entiers $n$ de~$S(x, y)$ sauf  au plus $$\ll\e^{-\cte{mk}\sqrt{\ubar}}\Psi(x, y)$$ exceptions. De plus, pour~$0 < Z \leq \cte{m4k}\,\ubar^{1/4}$, nous avons
\begin{equation}
m_{4, n} \ll (\log x)^4/\ubar^{2}Z^4 \label{eq:taille-Mk-ps14}
\end{equation}
sauf pour au plus~$\ll\e^{-\cte{m4k} \sqrt{\ubar}/(Z+1)}\Psi(x, y)$ entiers exceptionnels de~$S(x, y)$.
\end{coro}
\begin{proof}
  Pour $k=2$ ou 4, les fonctions $n\mapsto m_{k, n}$ satisfont aux hypothèses du Lemme~\ref{lemme:mino-fk} et de la Proposition~\ref{prop:majo-fk}.
Le résultat annoncé en découle directement.
\end{proof}

Posons
\begin{equation}
\Yu=\Yu(x, y) :=\Big \{ n \in S(x,y) : \sqrt{x} \leq n \leq x \text{ et } \omega_{x, y}(n) \geq \cte{wt} \ubar\Big \}\label{eq:def-Yu}
\end{equation}
où $\cte{wt}$ est la constante apparaissant au Lemme \ref{lemme:mino-fk}. Il résulte alors de cet énoncé et des estimations
$$  \Psi\big(\sqrt{x}, y\big) \ll x^{-\alpha/2} \Psi(x, y),\qquad \alpha\log x \gg \ubar, $$
que, quitte à modifier la valeur de~$\cte{wt}$, nous avons
$$ \mes\big(\Yu\big) = 1 + O\big(\e^{-\cte{wt}\ubar}\big). $$
Définissons encore
\begin{equation}
\Yd=\Yd(x, y) := \bigg\{ n \in \Yu(x, y) : \cte{Yd} \leq \frac{m_{2, n}\ubar}{(\log x)^2} \leq \frac1{\cte{Yd}},\, \frac{\cte{Yd}}\ubar \leq \frac{m_{4, n}\ubar^3}{(\log x)^4} \leq \frac1{\cte{Yd}} \bigg\},\label{eq:def-Yd}
\end{equation}
en choisissant~$\cte{Yd}$ assez petite, de sorte que le Lemme~\ref{lemme:mino-fk} et la Proposition~\ref{prop:majo-fk} impliquent
$$ \mes\big(\Yd\big) = 1 + O\Big(\e^{-\cte{Yd}\sqrt{\ubar}}\Big). $$
Observons que, pour~$n\in \Yd$, nous avons
\begin{equation}
\label{encwn}
1 \ll w_n \ll \ubar.
\end{equation}


\section{Estimations auxiliaires}

Posons
\begin{equation}
Z_n(s) := \bbE\big(\e^{sD_n}\big) = \prod_{p^\nu \| n} \frac{p^{(\nu+1)s}-1}{(\nu+1)(p^{s}-1)} \qquad (s\in\bbC).\label{eq:def-Zn}
\end{equation}
La formule de Perron fournit alors la représentation
\begin{equation}
\label{repr-Perron} \Prob\Big(D_n \geq \tfrac12\log n + z\sigma_n\Big) = \frac{1}{2\pi i} \int_{\beta - i\infty}^{\beta + i\infty} \frac{Z_n(s)}{n^{s/2}\e^{z\sigma_n s}s}\,\dd s
\end{equation}
pour tous~$\beta>0$,~$n\in \bbN^*$ et~$z\in\big[0, (\log n)/(2\sigma_n)\big[$ tels que~$n^{1/2}\e^{z\sigma_n}\not\in\bbN^*$.

\subsection{Propriétés de la série génératrice}
Pour $n\geqslant 1$, posons $ T_n := \max_{p^\nu \| n} (\nu+1)\log p $.
La fonction entière~$s\mapsto Z_n(s)$ ne s'annulant pas sur l'ouvert étoilé
$$ \cU_n = \bbC \smallsetminus \Big(\Big]-i\infty, -\frac{2\pi i}{T_n}\Big] \cup \Big[\frac{2\pi i}{T_n}, i\infty\Big[ \Big), $$
nous pouvons définir une branche du logarithme complexe
\begin{equation}\label{eq:def-vphi}
\vphi_n(s) := \log Z_n(s) \qquad (s\in \cU_n)
\end{equation}
telle que $\vphi_n(0)=0$.
\begin{lemme}\label{lemme:props-vphi}
Pour~$n\in \Yu(x, y)$ et~$\beta\geq 0$, les relations suivantes sont valides
\begin{enumerate}
\item\label{item:vphi2-pos} \quad$\vphi_n''(\beta)> 0$,
\item\label{item:vphi2-0}\quad $\vphi_n''(0) = \sigma_n^2=m_{2, n}$,
\item\label{item:vphi2-loin} \quad$\vphi_n''\big(\ubar/\log x\big) \gg (\log x)^2/\ubar$,
\item\label{item:vphi3} \quad$0 \leq -\vphi_n'''(\beta) \ll  \beta m_{4, n}$,
\item\label{item:vphi4} \quad $\dsp\sup_{T_n|s|\leqslant 1}|\vphi_n^{(4)}(s)| \ll m_{4, n}$.
\end{enumerate}
\end{lemme}
\begin{proof}
Nous avons 
$$ \vphi_n(s) = -\log \tau(n) + \sum_{p^\nu \| n}\log \bigg(\frac{p^{(\nu+1)s}-1}{p^s-1}\bigg)\qquad (s\in\cU_n), $$
d'où
\begin{equation}
\label{phij}
\varphi_n^{(j)}(s)=\sum_{\pnu\|n}\big\{g_j\big((\nu+1)\log p\big)-g_j(\log p)\big\}\qquad (j\geqslant 1)
\end{equation}
avec
\begin{equation}
\label{gjv}
\begin{aligned}
g_1(v;s)&:=\frac{v}{1-\e^{-vs}},\quad g_2(v;s):=\frac{-v^2\e^{-vs}}{(1-\e^{-vs})^2},\quad g_3(v;s):=\frac{-v^3\e^{-vs}(1+\e^{-vs})}{(1-\e^{-vs})^3}\\
g_4(v;s)&:=\frac{-v^4\e^{-vs}(1+4\e^{-vs}+\e^{-2vs})}{(1-\e^{-vs})^4}\cdot
\end{aligned}
\end{equation}
Notons également que, pour $|vs|<2\pi$, $s\neq0$, nous avons
\begin{equation}
\label{bernoulli}
g_1(v;s)=\sum_{k\geqslant 0}\frac{(-1)^k}{k!}B_kv^ks^{k-1},
\end{equation}
où $\{B_k\}_{k=0}^{\infty}$ est la suite des nombres de Bernoulli.
L'assertion (i) résulte imédiatement du fait que $v\mapsto g_2(v;s)$ est croissante sur $\bbR^+$ lorsque $s$ est réel positif. On obtient (ii) par dérivation de \eqref{bernoulli} par rapport à $s$ puis passage à la limite en $s=0$. Nous avons de plus
$$g'_2(v;s)=\frac{\dd g_2}{\dd v}(v;s)=\frac{v\e^{-vs}\big\{2(1-\e^{-vs})+vs(1+\e^{-vs})\big\}}{(1-\e^{-vs})^3}$$
de sorte que le terme général de \eqref{phij} est $\asymp 1/s^2$ lorsque $j=2$ et $\nu\log p\asymp1$. Il en va ainsi pour $s:=\ubar/\log x$ lorsque $\pnu$ est compté dans  la quantité $\omegat_{x, y}(n)$ définie en~\eqref{eq:def-omegat}. Il s'ensuit que, pour tout entier $n$ de $\Yu$, défini en \eqref{eq:def-Yu}, nous avons
$$ \vphi_n''\Big(\frac{\ubar}{\log x}\Big) \gg \Big(\frac{\log x}{\ubar}\Big)^2 \omegat_{x, y}(n) \gg \frac{(\log x)^2}\ubar $$
où~$\omegat_{x, y}(n)$ est défini en~\eqref{eq:def-omegat}. Cela établit l'assertion~\ref{item:vphi2-loin}.\par 
On vérifie que $v\mapsto g_3(v;s)$ est  décroissante sur $\bbR^+$: cela fournit le signe de $\varphi_n'''(\beta)$. Pour obtenir les majorations indiquées aux points (iv) et (v), nous observons simplement que \eqref{bernoulli} implique par dérivation que $g_3'(v;s)-2/s^3\ll v^3$, $g_4(v;s)-6/s^4\ll v^4$ lorsque $|vs|\leqslant 1$, $s\neq0$, et que \eqref{gjv} implique $g'_3(v;s)\ll v^3$ lorsque $s$ est réel, $sv>1$. 
\end{proof}

\subsection{Propriétés du point-selle}
Ce paragraphe est dévolu à l'estimation de certaines quantités apparaissant  naturellement dans le processus d'évaluation de l'intégrale de Perron \eqref{repr-Perron}.
 Rappelons la définition
$$ \sigma_n^2 :=\tfrac1{12} \sum_{p^\nu \| n} \nu(\nu+2)(\log p)^2. $$
Pour tout nombre réel~$z$ de 
$\big[0,(\log n)/(2\sigma_n)\big]$, nous définissons
le point-selle~$\beta_n = \beta_n(z)$ de l'intégrande de \eqref{repr-Perron} comme l'unique solution réelle positive de l'équation
\begin{equation}
\vphi_n'(\beta_n)=\sum_{p^\nu \| n}\bigg(\frac{\log p}{p^{\beta_n} - 1} - \frac{(\nu+1)\log p}{p^{(\nu+1)\beta_n} - 1} \bigg) =  z\sigma_n+\tfrac12\log n, \label{eq:def-pt-selle}
\end{equation}
où le membre de gauche est prolongé par continuité en~$\beta_n=0$. \par 
Considérée comme fonction de~$z$, la quantité $\beta_n=\beta_n(z)$ vérifie l'équation différentielle
\begin{equation}
\beta_n' \vphi_n''(\beta_n) = \sigma_n. \label{eq:eqdiff-beta}
\end{equation}
\begin{lemme}\label{lemme:mino-betap}
Nous avons
\begin{equation}
\label{minbeta'}
 \beta_n'(z) \geq 1/\sigma_n\qquad \bigg(0\leqslant z\leqslant (\log n)/2\sigma_n\Big). 
 \end{equation}
\end{lemme}
\begin{proof}
D'après le Lemme~\ref{lemme:props-vphi}, la fonction~$\vphi_n''$ est décroissante, strictement positive, et satisfait $\varphi_n''(0)=\sigma_n^2$. Cela implique~$\vphi_n''(\beta_n) \leq \sigma_n^2$, d'où l'inégalité annoncée en reportant dans~\eqref{eq:eqdiff-beta}.
\end{proof}

\begin{lemme}\label{lemme:taille-beta}
Il existe une constante absolue~$\cte{taillez}>0$, telle que l'on ait 
\begin{equation}
\label{ogbeta'} \beta_n'(z) \asymp \sqrt{\ubar} / \log x \qquad \Big(n\in \Yd(x, y),\, 0 \leq z \leq \cte{taillez}\sqrt{\ubar}\Big). 
\end{equation}
En particulier, nous avons dans les mêmes conditions $\beta_n(z) \asymp z\sqrt{\ubar}/\log x$.
\end{lemme}
\begin{proof}
La minoration incluse dans \eqref{ogbeta'} résulte  immédiatement de \eqref{minbeta'} et de la première relation \eqref{eq:taille-Mk-ps12} puique~$\sigma_n^2 = m_{2, n}$. Pour établir la majoration,  introduisons 
$$ z^* := \inf\Big\{z\in\big[0, (\log n)/2\sigma_n\big[ : \beta_n(z) >\ubar/\log x\Big\}, $$
de sorte que 
$$ \beta_n(z) \leq \ubar/\log x \qquad (0\leqslant z\leq z^*),$$
donc~$\vphi_n''(\beta_n)\geq \vphi_n''(\ubar/\log x) \gg (\log x)^2/\ubar$ en vertu de la décroissance de la fonction $\varphi_n''$ et du Lemme~\ref{lemme:props-vphi}(iii). Par~\eqref{eq:eqdiff-beta}, il suit
$$ \beta_n'(z) \ll \sigma_n\ubar / (\log x)^2 \qquad (0\leq z\leq z^*), $$
et donc~$\beta_n'(z) \ll \sqrt{\ubar}/\log x$, d'après \eqref{eq:taille-Mk-ps12}. Par intégration, nous obtenons  $$\ubar\leqslant \beta_n(z^*)\log x\ll z^*\sqrt{\ubar}$$ et par conséquent
$ z^* \gg \sqrt{\ubar}, $
ce qui conclut la démonstration.
\end{proof}
\par 
Nous sommes à présent en mesure de fournir une formule asymptotique précise pour~$\beta_n(z)$ valable pour les petites valeurs de $z$. Rappelons la définition
$$ w_n := m_{2, n}^2 / m_{4, n}=\sigma_n^4/m_{4,n}.$$
et l'encadrement \eqref{encwn}.
\begin{lemme}\label{lemme:estim-prec-beta}
Pour tous~$n\in \Yd(x, y)$, $0\leq z\leq \cte{taillez} \sqrt{\ubar}$, nous avons
\begin{align*} 
\beta_n(z)& = \frac{z}{\sigma_n}\bigg\{1 + O\bigg(\frac{z^2}{w_n}\bigg)\bigg\},\qquad 
 \vphi_n''\big(\beta_n(z)\big) = \sigma_n^2\bigg\{1 + O\bigg(\frac{z^2}{w_n}\bigg)\bigg\}. 
 \end{align*}
 \end{lemme}
\begin{proof} 
 Par dérivation, l'équation  \eqref{eq:eqdiff-beta} implique $\beta_n''\varphi''(\beta_n)+\beta_n'^2\varphi_n'''(\beta_n)=0$ et donc
$$ \beta_n'' = -\vphi_n'''(\beta_n) \beta_n'^3/\sigma_n. $$
D'après le Lemme~\ref{lemme:props-vphi}\ref{item:vphi3}, \eqref{ogbeta'} et la définition de $\Yd$ en \eqref{eq:def-Yd}, il s'ensuit que, sous les conditions de l'énoncé,
$$ \beta_n''\ll \frac{ \beta_n'^3  \beta_nm_{4, n}}{\sigma_n}\ll \frac{z}{\sigma_n w_n}\cdot $$
Par intégration, nous pouvons donc écrire
$$ \beta_n'(z) = \beta_n'(0) + \int_0^z \beta_n''(t) \dd t = \frac1{\sigma_n}\Big\{1 + O\Big(\frac{z^2}{w_n}\Big)\Big\}. $$
L'estimation annoncée pour $\beta_n$ résulte d'une intégration supplémentaire par rapport à~$z$. La seconde formule  est obtenue similairement à partir de la relation
$$ \vphi_n''\big(\beta_n(z)\big) = \vphi_n''(0) + \int_0^z \beta_n'(t)\vphi_n'''\big(\beta_n(t)\big)\dd t, $$
en utilisant \eqref{ogbeta'} et les points~\ref{item:vphi2-0}, \ref{item:vphi3} du Lemme~\ref{lemme:props-vphi}.
\end{proof}

\subsection{Décroissance dans les bandes verticales}
Nous utilisons la notation classique
$$ \|\vth\| := \inf_{k\in\bbZ}|\vth-k| \qquad (\vth\in\bbR)$$
pour désigner la distance d'un nombre réel $\vartheta$ à l'ensemble des entiers. Rappelons par ailleurs la définition de $\omegat_{x, y}(n)$ en~\eqref{eq:def-omegat}.
\par 
\begin{lemme}
\label{lemme:decr-n}
Sous les conditions $$0\leq\beta \leq \ubar/\log x,\quad \tau\in\bbR,\quad2\leq y \leq x,\quad n\in\bbN^*,$$ nous avons
\begin{equation}
\abs{\frac{Z_n(\beta + i\tau)}{Z_n(\beta)}} \leq \exp\bigg\{ -\cte{decr}\sum_{p|n}\NO{\frac{\tau\log p}{2\pi}}^4\bigg\}.\label{eq:majo-decr-grandp}
\end{equation}
Si, de plus, $|\tau| \leq 1/\log y$, alors
\begin{equation}
\abs{\frac{Z_n(\beta + i\tau)}{Z_n(\beta)}}\leq\bigg(1+\frac{(\tau\log x)^2}{\ubar^2}\bigg)^{-\cte{decr}\omegat_{x, y}(n)} .\label{eq:majo-decr-petitp}
\end{equation}
\end{lemme}
\begin{proof}
La première majoration découle directement des calculs de~\cite[formule~(2.26)]{Ultrafriables}, en remarquant que, dans notre cas,
$ \beta\log p \ll \ubar/u \ll 1. $
\par 
Examinons à présent le cas $|\tau| \leq 1/\log y$. Quitte à considérer la quantité conjuguée,  nous pouvons supposer sans perte de généralité que ~$\tau\geqslant 0$. Nous avons
\begin{equation}\label{eq:Zn-U}
\abs{\frac{Z_n(\beta+i\tau)}{Z_n(\beta)}}^2 = \prod_{p^\nu\|n} U_{p^\nu}(\beta,\tau) \leq \prod_{\substack{p^\nu\|n \\ \sqrt{y}\leq p\leq y\\ u/(2\ubar) \leq \nu \leq 2u/\ubar}} U_{p^\nu}(\beta, \tau),
\end{equation}
avec
\begin{align*}
U_{p^\nu}(\beta,\tau): =\ & \abs{\frac{1 + p^{-(\beta+i\tau)} + \cdots + p^{-\nu(\beta+i\tau)}}{1 + p^{-\beta} + \dotsc + p^{-\nu\beta}}}^2 =\frac{1+\{\sin((\nu+1)\tau_p)/\sinh((\nu+1)\beta_p)\}^2}{1+\{(\sin\tau_p)/\sinh\beta_p\}^2},
\end{align*}
où l'on a posé $\tau_p:=\dm\tau\log p$, $\beta_p:=\dm\beta\log p$.\par 
\smallskip
Nous estimerons cette quantité en utilisant l'inégalité
\begin{equation}
\label{ineg-sin}
\abs{\frac{\sin((\nu+1)\tau_p)}{(\nu+1)\sin\tau_p}}\leqslant1-\tfrac23\min\bigg(1,(\nu+1)^2\NO{\frac{\tau_p}{\pi}}^2\bigg),
\end{equation}
établie au lemme 1 de \cite{Tenenbaum-extremal}. 
Observons également que, puisque $(\nu+1)\beta_p\ll (u/\ubar)\beta\log y\ll1$, nous avons 
\begin{equation}
\label{ineg-sh}
\frac{(\nu+1)\sinh\beta_p}{\sinh((\nu+1)\beta_p)}\leqslant 1-\cte{beta}\nu^2\beta_p^2
\end{equation}
pour tous les couples $(p,\nu)$ du membre de droite de \eqref{eq:Zn-U}. 
\par 
Si $\tau_p\leqslant \dm\pi/(\nu+1)$, nous déduisons de \eqref{ineg-sin} et \eqref{ineg-sh} que 
$$U_{p^\nu}(\beta,\tau)\leqslant1-\frac{\cte{up0}(\sin\tau_p)^2\nu^2(\beta_p^2+\tau_p^2)}{(\sin\beta_p)^2+(\sin\tau_p)^2} \leqslant 1-\frac{\cte{up1}\tau^2(\log x)^2}{\ubar^2}\cdot$$
Cette majoration est bien compatible avec \eqref{eq:majo-decr-petitp}.
\par 
Si $\tau_p\geqslant \dm\pi/(\nu+1)\gg\beta_p$, 
nous notons que $0\leqslant \tau_p\leqslant \dm$ pour tous les termes considérés. Nous pouvons donc déduire de \eqref{ineg-sin} et \eqref{ineg-sh} que
\begin{equation}
\label{majUp}
\begin{aligned}
U_{\pnu}(\beta,\tau)&\leqslant \frac{1+(\sfrac56)^2(1-\cte{beta}\nu^2\beta_p^2)(\sin\tau_p/\sinh\beta_p)^2}{1+(\sin\tau_p/\sinh\beta_p)^2}\\
&\leqslant\frac{1}{1+\cte{up2}\tau_p^2\nu^2}\leqslant \bigg(\frac{1}{1+\tau^2(\log x)^2/\ubar^2}\bigg)^{\cte{up3}},
\end{aligned}
\end{equation}
une majoration encore  compatible avec \eqref{eq:majo-decr-petitp}. 
\end{proof}

La majoration du Lemme~\ref{lemme:decr-n} sera exploitée au Lemme \ref{coro:decr} {\it infra}. Les deux énoncés suivants sont nécessaires à la preuve.
Nous posons
$$ V_y(s) := \frac{y^{1-s}-1}{1-s} \qquad (s\in \bbC,\, y\geq 1). $$
\begin{lemme}
Pour $x\geqslant y\gg\log x$, nous avons
\begin{equation}\label{eq:mino-TP-decr}
\Re\big\{3 V_y(\alpha) - 4V_y(\alpha+i\tau)+V_y(\alpha+2i\tau)\big\} \gg \frac{\tau^4\ubar\log y}{(1-\alpha)^4+\tau^4}.
\end{equation}
\end{lemme}
\begin{proof}
Le membre de gauche de \eqref{eq:mino-TP-decr} vaut
$$2 \int_1^y \{1-\cos(\tau\log t)\}^2\frac{\dd t}{t^\alpha} = \frac2{1-\alpha}\int_1^{y^{1-\alpha}}\bigg\{1-\cos\bigg(\frac{\tau\log v}{1-\alpha}\bigg)\bigg\}^2\dd v. $$
Comme $y^{1-\alpha}\geqslant 1+\cte{ya}$ et $y^{1-\alpha}\gg(1-\alpha)\ubar\log y$ pour $y\gg\log x$,
il suffit donc de montrer que pour tout~$w\geq 1+\cte{w}$ et~$\vth\in\bbR_+$, nous avons
$$ J(w, \vth) := \int_1^w \|\vth\log v\|^4 \dd v \gg \frac{w \vth^4}{1+\vth^4}\cdot $$
La minoration étant triviale lorsque~$|\vth|\log w \leq \dm$,  nous pouvons supposer~$\vth\log w \geq \dm$. Suivant le raisonnement de~\cite[lemme~5.12]{RT}, posons~$\nu := \eta\min\{\vth, \dm\}$ où~$\eta\in(0, 1]$ est un paramètre à optimiser et écrivons
$$ J(w, \vth) \geq \nu^4\big(w-1-W\big), \qquad  W := \int_1^w \bfUn_{\{\|\vth\log v\| < \nu\}} \,\dd v. $$
L'hypothèse~$\vth\log w\geq \dm$ implique
$$ W \ll \frac\nu\vartheta\sum_{0\leq k\leq \vth\log w + \nu} \e^{k/\vth} \ll \frac{\eta\{w\e^{(\nu+1)/\vth}-1\}}{(\vth+1)(\e^{1/\vth}-1)} \ll \eta w.$$
Pour un choix adéquat de $\eta$ indépendant de $w$ et~$\vth$, il suit
$$ J(w, \vth) \geq \nu^4w\{1 + O(\eta)\} \gg w\vth^4/(1+\vth^4).$$
\end{proof}
\begin{lemme}\label{sec:lemme-mino-tp2}
Soit $\varepsilon>0$. Sous la condition~$1/\log y \leq |\tau| \leq \e^{(\log y)^{3/2-\ee}}$, nous avons
\begin{equation}
\label{eq:corodecr-obj-mino}
1 + \sum_{1<n\leq y} \frac{\Lambda(n)\{1-\cos(\tau\log n)\}^2}{n^{\alpha}\log n} \gg\frac{\ubar}{(\log 2\ubar)^4} + \bfUn_{[-1, 1]}(\tau) \log\big(1+|\tau|\log y\big).
\end{equation}
\end{lemme}
\begin{proof}
Lorsque~$1/\log y\leq |\tau|\leq 1$, une minoration par le second terme du membre de droite de \eqref{eq:corodecr-obj-mino} résulte immédiatement de \cite[lemme~III.4.13]{ITAN} en   majorant $\alpha$ par 1 et en restreignant la somme aux nombres premiers.\par 
Comme l'estimation cherchée est triviale lorsque $\ubar\ll_\varepsilon1$, nous pouvons nous limiter à estimer 
$$ \gT(y) := \sum_{n \leq y} \frac{\Lambda(n)\{1-\cos(\tau\log n)\}^2}{n^\alpha} $$
sous l'hypothèse supplémentaire que  $\ubar$ est assez grand en fonction de $\varepsilon$. 
Le lemme~6 de~\cite{HT}, et les calculs de~\cite[page 276]{HT} permettent alors d'écrire
\begin{align*}
\gT(y)& = \dm\Re\Big\{3 V_y(\alpha) - 4V_y(\alpha+i\tau) +  V_y(\alpha+2i\tau)\Big\} + O_\ee\bigg(\frac{1 + y^{1-\alpha}\e^{-(\log y)^{\ee/2}}} {1-\alpha}\bigg) \\
& \gg\frac{\tau^4\ubar\log y}{(1-\alpha)^4+\tau^4} + O_\ee\bigg(\frac{1 + y^{1-\alpha}\e^{-(\log y)^{\ee/2}}}{1-\alpha}\bigg)  \gg \frac{\ubar \log y}{(1-\alpha)^4(\log y)^4 + 1} \end{align*}
où la minoration résulte de~\eqref{eq:mino-TP-decr}.
On conclut en observant que~$(1-\alpha)\log y \asymp \log 2 \ubar$.
\end{proof}
\smallskip
Nous sommes à présent en mesure de déduire du Lemme~\ref{lemme:decr-n} une majoration en moyenne de la décroissance de $Z_n(s)$ dans les bandes verticales. Pour $\delta\in]0,1]$, nous posons
\begin{equation}
T_\delta=T_\delta(n;x,y) := \frac{\delta w_n^{1/4}\sqrt{\ubar}}{\log x},
\qquad T =T(x,y):= \min\Big\{\e^{\ubar}, \e^{(\log y)^{3/2-\ee}}\Big\},\label{eq:def-T0}
\end{equation}
\begin{coro}\label{coro:decr}
Pour une constante absolue convenable~$\cte{Yt}\in]0, \cte{taillez}[$ et tous~$2\leq y \leq x$, il existe un sous-ensemble~$\Yt=\Yt(x, y)$ de $S(x, y)$ satisfaisant à 
$$ \mes(\Yt) = 1 + O\Big(\e^{-\cte{Yt}\ubar/(\log 2\ubar)^4}\Big) $$
et, pour tous~$\delta\in]0, 1]$, $z\in [0, \cte{Yt}\sqrt{\ubar}]$, à
\begin{align}
\int_{T_{\delta}}^{1/\log y} &\ \abs{\frac{Z_n(\beta_n+i\tau)}{Z_n(\beta_n)}}\frac{\dd \tau}{\beta_n+\tau} \ll_\delta \frac{1}{(z+1)w_n}, && \big(n\in \Yt(x, y)\big), \label{eq:majodecr-I1} \\
\int_{1/\log y}^T &\ \abs{\frac{Z_n(\beta_n+i\tau)}{Z_n(\beta_n)}} \frac{\dd\tau}{\tau} \ll \e^{-\cte{Yt}\ubar/(\log 2\ubar)^4}, && \big(n\in \Yt (x, y)\big). \label{eq:majodecr-I2}
\end{align}
\end{coro}

\begin{proof}
D'après Le Lemme~\ref{lemme:taille-beta}, il est possible de
choisir~$\cte{Yt}\in]0, \cte{taillez}]$ de sorte que $\beta_n(\cte{Yt}\sqrt{\ubar}) \leq \ubar/\log x$ pour~$n\in\Yd(x, y)$.\par 
 Soit~$I_{1, z}(n)$ l'intégrale de~\eqref{eq:majodecr-I1}. Par~\eqref{eq:majo-decr-petitp} et~\eqref{eq:def-Yu},  pour~$z\in [0, \cte{Yt}\sqrt{\ubar}]$ et~$n\in \Yu(x, y)$, nous avons,
$$ I_{1, z}(n) \ll \int_{T_{\delta}(\log x)/\sqrt{\ubar}}^{\infty}  \frac{\dd\tau}{(1+\tau^2/\ubar)^{\cte{wt}\cte{decr}\ubar}\{\tau + (\log x)\beta_n/\sqrt{\ubar}\}} \cdot$$
En vertu de la seconde assertion du Lemme~\ref{lemme:taille-beta} et compte tenu de \eqref{eq:def-T0} et \eqref{encwn}, nous obtenons lorsque~$n\in \Yd(x, y)$,
$$ I_{1, z}(n) \ll \int_{\delta w_n^{1/4}}^{\infty} \frac{\dd\tau}{(1+\tau^2/\ubar)^{\cte{wt}\cte{decr} \ubar}(\tau+z)} \ll \frac{\e^{-\cte{I1zn}\delta w_n^{1/4}}}{(z+1)\delta w_n^{1/4}} + \e^{-\cte{I1zn}\ubar}.$$
 Cela implique bien~\eqref{eq:majodecr-I1}.

\smallskip

Considérons à présent l'intégrale~$I_{2, z}(n)$ apparaissant au membre de gauche de \eqref{eq:majodecr-I2}. Compte tenu de \eqref{eq:majo-decr-grandp}, l'estimation~\eqref{eq:majo-G-1} appliquée à la fonction multiplicative
$$ g(n):=\exp\bigg\{-\cte{decr}\sum_{p|n} \NO{\frac{\tau\log p}{2\pi}}^4 \bigg\} $$
fournit
\begin{equation}
\frac1{\Psi(x, y)}\sum_{n\in S(x, y)} \sup_{0\leq z \leq \cte{Yt}\sqrt{\ubar}}I_{2, z}(n) \ll \ubar \int_{1/\log y}^T \prod_{p\leq y}\bigg(1 - \frac{1-\e^{-\cte{decr}\|\tau(\log p)/(2\pi)\|^4}}{p^\alpha}\bigg)\dd\tau.\label{eq:corodecr-I2-pdt}
\end{equation}
Le dernier produit n'excède pas $\e^{-\cte{py}\gS(\tau)}$ avec
\begin{equation}
\gS(\tau):=\sum_{p\leq y} \frac{\{1-\cos(\tau\log p)\}^2}{p^{\alpha}}=\sum_{1<n\leq y} \frac{\Lambda(n)\{1-\cos(\tau\log n)\}^2}{n^{\alpha}\log n}+O\Big(1+y^{1/2-\alpha}\Big),\label{eq:corodecr-pdt-p}
\end{equation}
où l'estimation résulte du lemme~5 de~\cite{HT}. Le terme d'erreur étant trivialement~$\ll\ubar^{3/4}$, nous pouvons déduire de \eqref{eq:corodecr-obj-mino} que
$$
1+\gS(\tau)\gg \frac{\ubar}{(\log 2\ubar)^4} + \bfUn_{[-1, 1]}(\tau)\log(1+\tau\log y).
$$
En reportant dans~\eqref{eq:corodecr-pdt-p} puis~\eqref{eq:corodecr-I2-pdt}, nous obtenons
$$ \begin{aligned}
  \frac1{\Psi(x, y)}\sum_{n\in S(x, y)}\sup_{0\leq z \leq \cte{Yt}\sqrt{\ubar}} I_{2, z}(n) & \ll_\ee\e^{-2\cte{I2} \ubar/(\log 2\ubar)^4} \Bigg(\int_{1/\log y}^1 \frac{\dd \tau}{\tau(1+\tau\log y)^{\cte{I2fin}}} + \int_1^T \frac{\dd \tau}{\tau}\Bigg) \\
& \ll_\ee \e^{-\cte{I2} \ubar / (\log 2\ubar)^4},
\end{aligned} $$
puisque la première intégrale est bornée, alors que la seconde est~$O(\ubar)$. Cela implique bien le résultat annoncé, quitte à diminuer la valeur de~$\cte{Yt}$.
\end{proof}


\section{Preuve du Théorème~\ref{thm:TCL-cR}}

Rappelons les définitions des ensembles~$\Yd$ en~\eqref{eq:def-Yd} et~$\Yt$ au Corollaire~\ref{coro:decr}, et
posons
$$\Yq= \Yq(x, y) := \Yd(x, y) \cap \Yt(x, y), $$
de sorte que
\begin{equation}
\mes\big(\Yq\big) = 1 + O\Big(\e^{-\cte{Yq}\sqrt{\ubar}}\Big).\label{eq:mes-Yq}
\end{equation}
Dans toute la suite de ce paragraphe, nous restreignons les entiers $n$ considérés à parcourir~$\Yq(x, y).$

\subsection{Réduction}
Avec l'objectif de démontrer le Théorème~\ref{thm:TCL-cR}, fixons~$B>0$ arbitraire, et supposons donné~$z\in \bbR$ avec~$|z|\leq Bw_n^{1/4}$. Nous justifions tout d'abord que nous pouvons imposer la contrainte supplémentaire
\begin{equation} 
\label{redth1}
0\leq z \leq \min\Big\{\cte{Yt}\sqrt{\ubar}, Bw_n^{1/4}\Big\} .
\end{equation}
En effet, la condition~$z\leq \cte{Yt}\sqrt{\ubar}$ n'est effective que lorsque~$\ubar \leq (B/\cte{Yt})^2 w_n^{1/2}$. Puisque~$w_n \ll \ubar$ pour~$n\in \Yq(x, y)$, cela implique que~$\ubar$ et~$w_n$ sont bornés, auquel cas, l'estimation~\eqref{eq:TCL-cR} est trivialement satisfaite. Par ailleurs, la restriction à~$z\geq 0$ est justifiée par la symétrie de~$D_n$ par rapport à~$\dm$, qui implique$$ \Prob\Big(D_n \geq \tfrac12\log n - z\sigma_n\Big) = 1 -\Prob\Big(D_n> \tfrac12\log n + z \sigma_n\Big). $$

\subsection{Formule de Perron}
Ici et dans la suite,  posons
\begin{equation}
z_n := \tfrac12\log n + z \sigma_n.\label{eq:def-vnz}
\end{equation}
Une forme explicite de la formule de Perron (voir~\cite[th.~2.2, exercice~171]{ITAN}) permet d'écrire, pour~$T\geq 1$,
\begin{equation}
\Prob\Big(D_n \geq z_n\Big) = \frac1{2\pi i}\int_{\beta_n-iT}^{\beta_n+iT} Z_n(s)\e^{-z_ns}\frac{\dd s}s + O\big(R_T\big),\label{eq:prob-perron}
\end{equation}
avec
$$ R_T := \frac{Z_n(\beta_n) \e^{-z_n\beta_n}}{\sqrt T}\Bigg\{1 + \int_{-\sqrt T}^{\sqrt T} \abs{\frac{Z_n(\beta_n+i\tau)}{Z_n(\beta_n)}}\dd \tau \Bigg\}, $$
Choisissons~$T := \min\Big\{\e^{\ubar^{2/3}}, \e^{(\log y)^{4/3}}\Big\}$. Il résulte alors de~\eqref{eq:majodecr-I2} que
\begin{equation}
R_T \ll \frac{Z_n(\beta_n) \e^{-z_n\beta_n}}{\sqrt{T}}\Big\{1 + \sqrt{T}\e^{-\cte{Yt}\ubar/(\log 2\ubar)^4}\Big\} \ll \frac{Z_n(\beta_n)\e^{-z_n\beta_n}}{\sqrt{T}} \ll \frac{Z_n(\beta_n)\e^{-z_n\beta_n}}{\ubar^{3/2}} \cdot\label{eq:majo-R-perron}
\end{equation}

\subsection{Troncature}

Soit~$\delta\in ]0, 1]$ un paramètre à optimiser, et rappelons la définition de $T_\delta$ en~\eqref{eq:def-T0}. Comme $w_n\ll \ubar$ et~$z\ll \sqrt{\ubar}$ en vertu de \eqref{encwn} et \eqref{redth1}, nous déduisons du Corollaire~\ref{coro:decr} que
\begin{equation}
\label{eq:decoup-tronc}
\begin{aligned}
  \frac1{2\pi i}\int_{\beta_n-iT}^{\beta_n+iT} &Z_n(s)\e^{-z_n s}\frac{\dd s}s - \frac1{2\pi i}\int_{\beta_n-iT_{\delta}}^{\beta_n+iT_{\delta}} Z_n(s)\e^{-z_ns}\frac{\dd s}s 
\\
&\ll Z_n(\beta_n)\e^{-z_n\beta_n}\int_{T_{\delta}}^T \abs{\frac{Z_n(\beta_n+i\tau)}{Z_n(\beta_n)}}\frac{\dd \tau}{\tau+\beta_n}\\
& \ll_\delta Z_n(\beta_n)\e^{-z_n\beta_n}\Big\{\frac{1}{(z+1)w_n} + \e^{-\cte{Yt}\ubar/(\log 2\ubar)^4}\Big\} \ll_\delta \frac{Z_n(\beta_n)\e^{-z_n\beta_n}}{(z+1)w_n}.\\
\end{aligned}
\end{equation}

\subsection{Domaine de Taylor}

Rappelons la définition de~$\vphi_n(s)$ en~\eqref{eq:def-vphi}. Nous nous proposons ici d'estimer la contribution du segment $\beta_n+i\big[T_\delta,T_\delta\big]$ à l'intégrale de \eqref{eq:prob-perron}, i.e.
\begin{equation}\label{eq:domaine-taylor}
I_\delta := \frac1{2\pi i}\int_{\beta_n-iT_{\delta}}^{\beta_n+iT_{\delta}} Z_n(s)\e^{-z_ns}\frac{\dd s}s\cdot
\end{equation}
Supposons que~$s=\beta_n+i\tau$ avec~$|\tau|\leq T_{\delta}$. Comme $n\in\Yq$, nous avons 
$$ (\nu\log p)^4 \ll m_{4, n} \asymp (\log x)^4/(w_n\ubar^2)\qquad \big(\pnu\|n\big). $$
Il s'ensuit qu'en choisissant $\delta$ comme une constante absolue et suffisamment petite, nous avons ~$(\nu+1)|\tau|\log p \leq \dm$ pour tout $s=\beta_n+i\tau$ du segment d'intégration. De plus, quitte à diminuer la valeur de~$\cte{Yt}$ dans \eqref{redth1}, le Lemme~\ref{lemme:taille-beta} implique également~$(\nu+1)\beta_n\log p \leq \dm$. Le Lemme~\ref{lemme:props-vphi}\ref{item:vphi4} et l'encadrement \eqref{encwn} fournissent alors
\begin{equation}
\vphi^{(4)}(s) \ll m_{4,n} \ll  (\log x)^4/\ubar^{2}. \label{eq:majo-vphi4}
\end{equation}
De plus, nous pouvons déduire du Lemme~\ref{lemme:props-vphi}\ref{item:vphi3} et de \eqref{redth1} que
\begin{equation}
\vphi_n'''(\beta_n) \ll_B (\log x)^3/\big(\ubar^{3/2}w_n^{3/4}\big). \label{eq:majo-vphi3-2}
\end{equation}
Compte tenu de la définition de~$T_{\delta}$ en \eqref{eq:def-T0}, il vient
$$ T_{\delta}^4 \sup_{|\tau|\leq T_{\delta}}|\vphi_n^{(4)}(\beta_n+i\tau)|+T_{\delta}^3|\vphi_n'''(\beta_n)| \ll_B 1. $$
En reportant dans le développement de Taylor
$$ \vphi_n(s) = \vphi_n(\beta_n) + i\tau\vphi_n'(\beta_n) - \dm\tau^2\vphi_n''(\beta_n) - i\tfrac16\tau^3\vphi_n'''(\beta_n) + O\Big((\tau\log x)^4/\big(\ubar^2 w_n \big)\Big), $$
et en tenant compte de la définition du point-selle, nous obtenons
$$ Z_n(s)\e^{-z_n s} = Z_n(\beta_n)\e^{-z_n\beta_n-\tau^2\vphi_n''(\beta_n)/2}\Bigg\{1 + \lambda \tau^3 + O_B\Bigg(\frac{(\tau\log x)^6}{\ubar^3w_n^{3/2}} + \frac{(\tau \log x)^4}{\ubar^2w_n}\Bigg) \Bigg\} $$
où~$\lambda=\lambda(n, z)$. Intégrons cette estimation sur le segment~$-T_{\delta}\leqslant \tau\leqslant  T_{\delta}$ en notant que la contribution du terme de degré 3 est nulle par symétrie. Posant
$$ \mu_2 = \mu_2(n, z) := \sqrt{\vphi_n''(\beta_n)} \asymp (\log x)/ \sqrt{\ubar}, $$
 et effectuant le changement de variables $t:=\mu_2\tau$, il vient
\begin{align}
\label{Id1}
I_\delta = \frac{Z_n(\beta_n)\e^{-z_n\beta_n}}{2\pi} \int_{-\delta w_n^{1/4}}^{\delta w_n^{1/4}} \e^{-t^2/2}\Bigg\{1+O_B\bigg(\frac{t^4}{w_n} + \frac{t^6}{w_n^{3/2}}\bigg)\Bigg\} \frac{\dd t}{\beta_n\mu_2+it}\cdot
\end{align}
Rappelons la définition de~$\Phi(z)$  en~\eqref{eq:defs-Phi-wn}. Nous avons par inversion de Fourier
\begin{align*}
 \frac1{2\pi}\int_{-\infty}^\infty \frac{\e^{-t^2/2}\dd t}{\beta_n\mu_2+it} &= \e^{-\beta_n^2\mu_2^2/2}\Phi\big(\beta_n\mu_2\big) \asymp \frac1{1+\beta_n\mu_2}, \\
 \int_{X}^{\infty} \frac{t^k e^{-t^2/2} \dd t}{\beta_n\mu_2+t} &\ll \frac{(X^5+1)\e^{-X^2/2}}{1+\beta_n\mu_2} \qquad \big(X\geq 0, k=4\mbox{ ou }6\big). 
 \end{align*}
En reportant dans \eqref{Id1}, nous obtenons
\begin{equation}
I_\delta = \Big\{1 + O_B\Big(\frac1{w_n}\Big)\Big\} Z_n(\beta_n)\e^{-z_n\beta_n + \beta_n^2\mu_2^2/2} \Phi(\beta_n\mu_2). \label{eq:estim-taylor}
\end{equation}

\subsection{Conclusion}\label{sec:col-conclusion}

De~\eqref{eq:prob-perron}, \eqref{eq:majo-R-perron} et \eqref{eq:estim-taylor}, nous déduisons que sous l'hypothèse \eqref{redth1} et~pour $n\in \Yq(x, y)$, nous avons
\begin{equation}
\Prob\big(D_n\geq z_n\big) = \Big\{1 + O_B\Big(\frac1{w_n}\Big)\Big\} \e^{E_n(z)}\Phi(\beta_n\mu_2),\label{eq:approx-centre}
\end{equation}
avec
$$ E_n(z) := \vphi_n(\beta_n) - (\tfrac12\log n + z\sigma_n)\beta_n + \tfrac12\beta_n^2\vphi_n''(\beta_n). $$
Il résulte alors de~\eqref{eq:def-pt-selle} et \eqref{eq:eqdiff-beta} que
$$ E'_n(z) = \tfrac12 \beta_n' \beta_n ^2 \varphi_n'''(\beta_n),$$
et donc~$E'_n(z) \ll z^3/w_n$ d'après les Lemmes~\ref{lemme:props-vphi}\ref{item:vphi3} et~\ref{lemme:taille-beta}. Comme $E_n(0)=0$, il s'ensuit que
\begin{equation}
\label{majEn}
E_n(z) \ll z^4/w_n.
\end{equation}
À fins de référence ultérieure, nous notons que cette majoration est valable en toute généralité. Sous l'hypothèse $z^4\ll w_n$ de l'énoncé, le membre de droite est borné et il suit
\begin{equation}
\e^{E_n(z)}= 1 + O_B(z^4/w_n). \label{eq:estim-cA}
\end{equation}
Maintenant, nous déduisons du Lemme~\ref{lemme:estim-prec-beta}, que~$\beta_n\mu_2 = z\{1 + O(z^2/w_n)\}$, et donc, dans le domaine considéré,
\begin{equation}
\label{eq:estim-cB}\e^{-\beta_n^2\mu_2^2/2} \asymp_B \e^{-z^2/2},\quad \Phi(\beta_n\mu_2) = \{1 + O(z^2(1+z)/w_n)\}\Phi(z).
\end{equation}
En reportant~\eqref{eq:estim-cA} et~\eqref{eq:estim-cB} dans \eqref{eq:approx-centre}, nous obtenons, sous les conditions \eqref{redth1} et $n\in\Yq$,
\begin{equation}
\Prob\big(D_n \geq z_n\big) = \Phi(z)\bigg\{1 + O_B\bigg(\frac{1+z^4}{w_n}\bigg)\bigg\}.\label{eq:TCL-zpos}
\end{equation}
Cela complète la démonstration du Théorème \ref{thm:TCL-cR}.

\section{Estimation en moyenne : preuve du Théorème \ref{thm:Dxyz}}

Rappelons la définition en \eqref{def-Dxyz} de la fonction de répartition $\cD(x,y;z)$, où la quantité \mbox{$\bsigma=\bsigma(x,y)$} est le terme principal de la valeur moyenne friable de $\sigma_n^2$. Posant $z_n^*:=\dm\log n+z\bsigma$, nous avons donc
$$\cD(x,y;z)=\frac1{\Psi(x,y)}\sum_{n\in S(x,y)}\PP\big(D_n\geqslant z_n^*\big).$$
Pour établir l'estimation~\eqref{eq:estim-moy-2}, nous pouvons, 
comme précédemment, grâce à la symétrie de~$D_n$,   restreindre l'étude au cas $z\geq 0$.
\par 
 Soit~$\eta>0$ une constante assez petite.  Pour chaque valeur de~$z\leq \eta \ubar^{1/5}$, nous
introduisons les sous-ensembles 
\begin{align*}\Yc(z)&=\Yc(z;x,y):=\big\{n\in\Yq(x,y):w_n>1+z^4\big\},\\
 \Ys(j)&= \Ys(j;x,y):=\big\{n\in\Yq(x,y):2^j<w_n\leqslant 2^{j+1}\big\}\quad\big(1+z^2<2^j\leqslant 1+z^4\big),\\
 \Ysp(z)&=\Ysp(z;x,y)=\big\{n\in S(x,y):w_n\leqslant 1+z^2\big\}\cup\big(S(x,y)\sset\Yq(x,y)\big)
\end{align*}
D'après la Proposition~\ref{prop:taille-cR} et la relation \eqref{eq:mes-Yq}, nous avons, pour un choix convenable de $\eta$,
$$\nu_{x,y}\Big\{\Ysp(z)\Big\}\ll \e^{-\cte{Yq}\sqrt{\ubar}} + \e^{-c\sqrt{\ubar}/(1+\sqrt{z})}\ll\Phi(z)\e^{-\ubar^{2/5}}.$$
La contribution à $\cD(x,y;z)$ des entiers $n$ de $\Ysp(z)$ est donc englobée par le terme d'erreur de~\eqref{eq:estim-moy-2}.
\par 
Rappelons la définition de $z_n$ en \eqref{eq:def-vnz}, de sorte que $z_n^*-z_n=z(\bsigma-\sigma_n)$. Désignons par $\cD_j$ la contribution à $\cD(x,y;z)$ des entiers $n$ de $\Ys(j)$. Comme il résulte de la majoration de Rankin, de la formule~\eqref{majEn} et du Lemme~\ref{lemme:estim-prec-beta} que 
$$\PP(D_n\geqslant z_n^*)\leq Z_n(\beta_n)\e^{-z_n^* \beta_n} \ll (1+z)\e^{\cte{moyast}z^4/w_n}\Phi\big(z\bsigma/\sigma_n\big),$$
une application de l'inégalité de Cauchy--Schwarz permet d'écrire
\begin{equation}
\label{Dj1}
\cD_j^2\ll (1+z)^2\e^{2\cte{moyast}z^4/2^j}U_jV,
\end{equation}
avec
\begin{equation}
\label{def-UjV}
U_j:=\nu_{x,y}\big\{n\geqslant 1:w_n\leqslant 2^{j+1}\big\},\qquad V:=\frac{1}{\Psi(x,y)}\sum_{n\in S(x,y)} \Phi\bigg(\frac{z\bsigma}{\sigma_n}\bigg)^2.
\end{equation}
D'après la Proposition \ref{prop:taille-cR}, nous avons
$$U_j\ll\e^{-2\cte{uj}\sqrt{\ubar}/2^{j/4}}.$$
En reportant dans \eqref{Dj1}, il suit, pour un choix convenable de $\eta$,
\begin{equation}
\label{Dj2}
\cD_j^2\ll (1+z)^2 \e^{-(2\cte{moyast}z^4/2^{3j/4}-2\cte{uj}\sqrt{\ubar})/2^{j/4}} V \leq \e^{-\cte{uj}\ubar^{3/10}}V.
\end{equation}
Pour estimer $V$, nous majorons le terme général par $\Phi(z)$ si $\sigma_n\leqslant \bsigma$, et par 1 si $\sigma_n>2\bsigma$. Nous obtenons ainsi, en appliquant la Proposition \ref{prop:majo-fk} avec $k=2$, $h=0$, $\delta=3$, 
\begin{equation}
\label{V1}
V\ll \Phi(z)^2+\nu_{x,y}\big\{n\geqslant 1:\sigma_n^2>4\bsigma^2\big\}+W\ll\Phi(z)^2+W,
\end{equation}
avec
\begin{equation}
\label{majW}
W:=\frac{1}{\Psi(x,y)} \ssum{n\in S(x,y) \\ \bsigma<\sigma_n\leqslant 2\bsigma}\Phi\bigg(\frac{z\bsigma}{\sigma_n}\bigg)^2.
\end{equation}
Nous évaluons cette dernière quantité grâce à la Proposition \ref{prop:majo-fk}:
\begin{equation}
\label{majW}
\begin{aligned}
W&\ll\frac1{z+1}\int_{z/2}^\infty\e^{-t^2}\nu_{x,y}\big\{n\geqslant 1: \sigma_n>z\bsigma/t\big\}\,\dd t\\
&\ll\frac\ubar{z+1}\int_{z/2}^\infty\e^{-t^2-\cte{W}\max(0,(z-t)/t)\sqrt{\ubar}}\,\dd t\ll \frac{\ubar \e^{-z^2}}{(1+z)^2} \ll \ubar \Phi(z)^2 .
\end{aligned}
\end{equation}
En reportant dans \eqref{V1} puis \eqref{Dj2}, nous concluons que la contribution de $\cup_j\Ys(j)$ est également englobée dans le terme d'erreur de \eqref{eq:estim-moy-2}.
\par 
Pour les entiers $n$ de $\Yc$, nous pouvons appliquer le Théorème~\ref{thm:pcp-coro} avec $z_n^*$ au lieu de $z_n$.  Nous obtenons
\begin{equation}
\label{D1}
\cD(x,y;z)=\frac1{\Psi(x,y)}\sum_{n\in \Yc(x,y)}\bigg\{1+O\bigg(\frac{1+z^4}{w_n}\bigg)\bigg\}\Phi\bigg(\frac{z\bsigma}{\sigma_n}\bigg)+O\bigg(\frac{\Phi(z)}{\ubar}\bigg).
\end{equation}
Soit alors $\Delta:=K(\log \ubar)/\sqrt{\ubar}$, où $K$ est une constante absolue assez grande. La contribution à la dernière somme des entiers $n$ tels que $\sigma_n> (1+\Delta)\bsigma$ peut être estimée de la même manière que la quantité $W$ précédemment considérée, l'argument de l'exponentielle dans l'analogue de l'intégrale de \eqref{majW} comprenant à présent un terme  proportionnel à $-\max(\Delta,(z-t)/t)\sqrt{\ubar}$.
Nous obtenons ainsi que, pour un choix convenable de $K$, cette contribution  n'excède pas le terme d'erreur de \eqref{D1}. Ainsi\begin{equation}
\label{D2}
\cD(x,y;z)=\frac1{\Psi(x,y)}\ssum{n\in \Yc(x,y)\\\sigma_n\leqslant (1+\Delta)\bsigma}\bigg\{1+O\bigg(\frac{1+z^4}{w_n}\bigg)\bigg\}\Phi\bigg(\frac{z\bsigma}{\sigma_n}\bigg)+O\bigg(\frac{\Phi(z)}{\ubar}\bigg).
\end{equation}
\par 
Désignons par~$\cE$ la contribution du terme d'erreur~$O((1+z^4)/w_n)$ au membre de droite. Lorsque~$\sigma_n  \leq (1+\Delta) \bsigma$, l'hypothèse~$z\leq \eta \ubar^{1/5}$ implique~$(z\bsigma/\sigma_n)^2 \geq z^2 + O(1)$, donc
$$ \cE \ll \frac{(1+z^4)\Phi(z)}{\bsigma^4 \Psi(x, y)} \sum_{n\in S(x, y)} m_{4, n}\ll  \Phi(z)\frac{1+z^4}{\ubar},$$
où la somme en $n$ a été estimée par le théorème~2.9 de~\cite{BT05}, qui fournit immédiatement une borne $\ll\Psi(x,y)\bsigma^4/\ubar$.\par 
Nous pouvons donc énoncer que
\begin{equation}
\label{D2}
\cD(x,y;z)=\frac1{\Psi(x,y)}\ssum{n\in \Yc(x,y)\\\sigma_n\leqslant (1+\Delta)\bsigma}\Phi\bigg(\frac{z\bsigma}{\sigma_n}\bigg)+O\bigg(\Phi(z)\frac{1+z^4}{\ubar}\bigg).
\end{equation}
\par 
Nous allons voir que l'on peut remplacer  $ \Yc(x,y)$  par $S(x,y)$ dans les conditions de sommation. En effet, la somme relative à $S(x,y)\sset \Yc$ n'excède pas $\sqrt{VY}$,  où $V$ est défini en \eqref{def-UjV} et
$$Y:=\nu_{x,y}\big\{n\geqslant 1:w_n\leqslant 1+z^4\big\}\ll\e^{-c\ubar^{3/10}}$$
en vertu de la Proposition \eqref{prop:taille-cR}. Au vu de \eqref{V1} et \eqref{majW}, nous obtenons bien, comme annoncé,
\begin{equation}
\label{D3}
\cD(x,y;z)= \cD_0(x, y ; z)+O\bigg(\Phi(z)\frac{1+z^4}{\ubar}\bigg),
\end{equation}
avec
\begin{equation}
\label{eq:rel-cD-cD0}
 \cD_0(x, y ; z):=\frac1{\Psi(x,y)}\ssum{n\in S(x,y)\\\sigma_n\leqslant (1+\Delta)\bsigma}\Phi\bigg(\frac{z\bsigma}{\sigma_n}\bigg).
\end{equation}
\par 
L'estimation~\eqref{eq:estim-moy-2} étant trivialement satisfaite lorsque~$\ubar$ est borné, nous supposons sans perte de généralité que~$\Delta\leq 1/2$. Nous avons alors
$$ \cD_0(x, y ; z) = \frac1{\sqrt{2\pi}}\int_{z/(1+\Delta)}^{\infty} \mes\big\{n : \sigma_n \geq z\bsigma/t\big\} \e^{-t^2/2}\dd t. $$
Notant
$$ F(h) := \begin{cases} \mes(\{n : \sigma_n^2 \geq (1+h)\bsigma^2\}), & (h\geq 0), \\ \mes(\{n : \sigma_n^2 \geq (1+h) \bsigma^2\}) - 1, & (-1\leqslant h<0), \end{cases} $$
nous déduisons des Propositions~\ref{prop:mino-fk-precis} et~\ref{prop:majo-fk} que
\begin{equation}
F(h) \ll \ubar\e^{-\cte{Fdecr}\min\big(|h|,\sqrt{|h|}\big)\sqrt{\ubar}} .\label{eq:majo-Fdelta}
\end{equation}
 Après  interversion de sommation et changement de variables défini par~$z^2=t^2(1+h)$, nous obtenons
$$ \cD_0(x, y ; z) = \Phi(z) +  \frac{z}{2\sqrt{2\pi}}\int_{-1}^{\Delta_1} F(h) \e^{-z^2/(2+2h)} \frac{\dd h}{(1+h)^{3/2}} $$
avec~$\Delta_1 = \sqrt{1+\Delta}-1 \asymp \Delta$. L'intervalle d'intégration peut être remplacé par~$[-\Delta_1, \Delta_1]$ au prix d'une erreur~$\ll F(-\Delta_1)\Phi(z)\ll \Phi(z)/\ubar$, en vertu de~\eqref{eq:majo-Fdelta}, pour un choix convenable de la constante $K$. Lorsque~$|h|\leq \Delta_1$, nous avons~$z^2|h| \ll 1$. Il suit
$$ \cD_0(x, y; z) = \Phi(z)\bigg\{1 + O\Big(\frac1\ubar\Big)\bigg\} + \frac{z\e^{-z^2/2}}{2\sqrt{2\pi}}\int_{-\Delta_1}^{\Delta_1} F(h) \dd h + O\Bigg(z(1+z^2)\e^{-z^2/2}\int_{-1}^\infty h F(h) \dd h\Bigg). $$
L'intégrale apparaissant dans le terme d'erreur vaut
$$ \int_{-1}^\infty h F(h) \dd h = \frac1{2\Psi(x, y)} \sum_{n\in S(x, y)} \bigg(\frac{\sigma_n^2 - \bsigma^2}{\bsigma^2}\bigg)^2 \ll \frac1\ubar $$
d'après l'inégalité de Turán-Kubilius friable~\cite[th. 1.1]{BT-TK}. De plus, il résulte de \eqref{eq:majo-Fdelta} que
$$ \int_{-\Delta_1}^{\Delta_1} F(h) \dd h = \int_{-1}^{\infty} F(h)\dd h + O\Big(\frac1\ubar\Big) $$
lorsque~$K$ est choisie suffisamment grande. Enfin, nous pouvons écrire
$$ \int_{-1}^{\infty} F(h)\dd h  = \frac1{\Psi(x, y)}\sum_{n\in S(x, y)} \frac{\sigma_n^2 - \bsigma^2}{\bsigma^2}  \ll \frac1{\ubar} $$
grâce au théorème~2.9 de \cite{BT05}. \par 
En regroupant nos estimations, nous obtenons
$$ \cD_0(x, y ; z) = \Phi(z) \Big\{ 1 + O\Big(\frac{1+z^4}\ubar\Big)\Big\}, $$
ce qui, compte tenu de~\eqref{D3}, fournit bien l'estimation souhaitée \eqref{eq:estim-moy-2}.

\bibliographystyle{amsalpha-modif}
\bibliography{diviseurs}

\end{document}